\documentclass[a4paper,12pt,twoside]{amsart} 
\usepackage{amsmath,amsthm}
\usepackage{amsfonts}
\usepackage{amssymb}
\usepackage[english]{babel}

\newtheorem{theo}{Theorem}[section]
\newtheorem{lem}[theo]{Lemma}
\newtheorem{prop}[theo]{Proposition}
\newtheorem{cor}[theo]{Corollary}

\renewcommand{\P}{{\mathbb P}}

\newcommand{\B}{{\bf B}}

\newcommand{\e}{{\text{e}}}

\theoremstyle{remark}

\begin{document}

\baselineskip=14pt

\title[Joint and double coboundaries of commuting contractions]
{Joint and double coboundaries \\ of commuting contractions}

\author{Guy Cohen}
\address{Department of Electrical 
Engineering, Ben-Gurion University, Beer-Sheva, Israel}
\email{guycohen@bgu.ac.il}

\author{Michael Lin}
\address{Department of Mathematics, Ben-Gurion University, Beer-Sheva, Israel}
\email{lin@math.bgu.ac.il}

\subjclass[2010]{Primary: 47A35, 47A10, 37A05, 37A30. Secondary: 37A45, 47B15}
\keywords{commuting contractions, double coboundaries, joint coboundaries, $\mathbb Z^2$ 
actions, maximal spectral type, ergodic circle rotations, Diophantine approximation, 
joint spectrum}

\begin{abstract}
Let $T$ and $S$ be commuting contractions on a Banach space $X$. The elements of 
$(I-T)(I-S)X$ are called {\it double coboundaries}, and the elements of 
$(I-T)X \cap (I-S)X$ are called {\it joint cobundaries}. For $U$ and $V$ the unitary 
operators induced on $L_2$ by commuting invertible measure preserving transformations
which generate an aperiodic $\mathbb Z^2$-action, 
we show that there are joint coboundaries in $L_2$ which are not double coboundaries. 
We prove that if $\alpha$,$\beta \in (0,1)$ are irrational, with $T_\alpha$ and $T_\beta$ 
induced on $L_1(\mathbb T)$ by the corresponding rotations, then there are joint 
coboundaries in $C(\mathbb T)$ which are not measurable double cobundaries (hence not double 
coboundaries in $L_1(\mathbb T)$).
\end{abstract}

\maketitle
\vspace*{-0.8 cm}

\section{Introduction}

Let $\alpha$ be irrational, and let $\theta_\alpha x =x+\alpha \mod 1$ for $x \in[0,1)$. 
Then $\theta_\alpha$ preserves Lebesgue's measure, and the operator $T_\alpha h=h\circ \theta$
defines an invertible isometry on all the spaces $L_p[0,1)$, $\ 1\le p \le \infty$.
Motivated by Euler's formal approach to Fourier series, Wintner \cite{Wi} studied  the 
existence of solutions $g \in L_1$ of the equation $(I-T_\alpha)g=f$ for a given $f \in L_1$ 
(or $L_2$). The translation $\theta_\alpha$ corresponds to the rotation 
$z= \e^{2\pi i x} \to  \e^{2\pi i(x+\alpha)}=\e^{2\pi i\alpha}z$ of the unit circle $\mathbb T$. 
This rotation is minimal (all orbits are dense in $\mathbb T$), since $\alpha$ is irrational. 

Gottschalk and Hedlund \cite[p. 135]{GH} proved that  if $\theta$ is a minimal homeomorphism
of a compact Hausdorff space $K$, then a continuous function $ f$ is of the form 
$f=g-g\circ\theta $ for some continuous $g$ if and only if 
$\sup_n \|\sum_{k=0}^{n-1}f \circ \theta^k \|_{C(K)} < \infty$. Browder \cite{Br} proved
that if $T$ is a power-bounded operator on a reflexive Banach space $X$, then
\begin{equation} \label{browder}
y\in (I-T)X \qquad \text{if and only if} \qquad \sup_n \|\sum_{k=0}^{n-1} T^k y\| <\infty.
\end{equation}
Lin and Sine \cite[Theorem 7]{LS} proved (\ref{browder}) for {\it contractions} of $L_1$.

When the equation $(I-T)x=y$ (for $y$ given) has a solution, i.e. $y \in (I-T)X$,
 $y$ is called a {\it coboundary}. Note that when $T$ is induced by a measure preserving 
transformation and $f \in L_p$, the solution of $(I-T)g=f$ may be in a larger space (e.g. 
$f \in (I-T)L_1$), or even measurable and non-integrable (and then $f$ is called a 
{\it measurable} coboundary).

Recently, Adams and Rosenblatt \cite{AR1} studied the following problem: let $(\Omega,\P)$
be a standard probability space; given $f \in L_p(\P)$, is there some ergodic invertible
measure preserving transformation $\theta$ such that $f =g -g\circ\theta$ for some $g$,
and what are the integrability properties of $g$?

We refer to the introduction of \cite{CL} for additional discussion
of developments following the results of Gottschalk-Hedlund and of Browder.
\smallskip

A two-dimensional extension of Browder's result was obtained by the present authors in 
\cite[Theorem 3.1]{CL}:
Let $T$ and $S$ be {\it commuting} contractions on a reflexive Banach space $X$. Then
\begin{equation} \label{CL}
y \in (I-T)(I-S)X \qquad \text{if and only if} \qquad 
\sup_n \big\|\sum_{k=0}^{n-1}\sum_{j=0}^{n-1}T^kS^jy\big\| <\infty.
\end{equation}
The elements of $(I-T)(I-S)X$ were called in \cite{CL} {\it double coboundaries}. 
Clearly double coboundaries are in $(I-T)X \cap (I-S)X$ (i.e. are {\it joint (common) 
coboundaries}).  

The paper deals with the existence of joint coboundaries  (of commuting
contractions) which are not double coboundaries.
We mention that Adams and Rosenblatt \cite{AR} studied the {\it existence} of 
joint coboundaries of {\it non-commuting} contractions. 
\smallskip

Double and joint cobundaries can be interpreted in terms of rates of convergence in 
mean ergodic theorems. Denote $A_n(T):= \frac1n\sum_{k=0}^{n-1} T^k$. Then $A_n(T)x \to 0$
if and only if $x \in \overline{(I-T)X}$, and Browder's theorem means that the rate is
$1/n$ if and only if $x$ is a coboundary. When $X$ is reflexive, 
$(1/n^2)\sum_{j=0}^{n-1}\sum_{k=0}^{n-1} T^jS^k = A_n(T)A_n(S)$ converges strongly, and
$A_n(T)A_n(S)y \to 0$ if and only if $y \in \overline{(I-T)X+(I-S)X}$ (see Proposition
\ref{squares-se}); (\ref{CL}) means that the rate is $1/n^2$ if and only if $y$ is a 
double coboundary. The question becomes whether rates of $1/n$ in the convergence to zero of 
$A_n(T)y$ and $A_n(S)y$ imply the rate of $1/n^2$ for $A_n(T)A_n(S)y \to 0$.
\smallskip

Let $P$ be an ergodic Markov operator on a general state space $(\mathcal S,\Sigma)$,
with invariant probability $\pi$, and let $(\xi_k)$ be the induced stationary Markov chain
on $(\Omega, \mathcal B,\mathbb P_\pi)$. Gordin and Lifshits \cite{GL} proved a central limit 
theorem for $f(\xi_k)$ when $f$ is a coboundary of $P$ on $L_2(\mathbb S,\pi)$. Extensions
to central limit theorems for random fields lead to the study of double coboundaries,
which play a role in obtaining martingale-coboundary decompositions \cite{Go} 
(see also \cite{Vo} and references therein).
\medskip

Let $\mathbb S$ be a topological semi-group, and let $\mathbf R(s)$ be a bounded representation 
of $\mathbb S$ by linear operators on a Banach space $X$. A {\it cocycle} for $\mathbf R$ is a
function $F: \mathbb S \longmapsto X$ satisfying
$$
F(s_1s_2)= F(s_1) +\mathbf R(s_1)F(s_2) \quad \text{ for } \ s_1,s_2 \in \mathbb S.
$$
$F: \mathbb S \longmapsto X$ is called a {\it coboundary} if there exists $x \in X$ such that
$F(s) =(I- \mathbf R(s))x$ for every $s \in \mathbb S$.  Parry and Schmidt \cite{PS} proved that
when $X$ is reflexive and $\mathbb S$ is Abelian, a cocycle is a coboundary if and only if 
it is bounded; when $\mathbb S =\mathbb N$, we recover (\ref{browder}),
since then cocycles are of the form $F(n)=\sum_{k=0}^{n-1} T^k F(1)$.
\smallskip

The following observation by Y. Derriennic (see also \cite{BD}) relates 
our definitions to classical cocycles of representations of $\mathbb N^2$.
\begin{prop} \label{yves}
Let $T$ and $S$ be commuting contractions of a Banach space $X$, and let 
$\mathbf R(\vec u):=T^{u_1}S^{u_2}$ for $\vec u=(u_1,u_2) \in \mathbb N^2$.

(i) If $F(\vec u)$ is a cocycle for $\mathbf R$, then $(I-S)F(\vec e_1)=(I-T)F(\vec e_2)$
is a joint coboundary.

(ii) If $z=(I-T)y=(I-S)x$ is a joint coboundary, then there exists a cocycle $F$ for 
$\mathbf R$ with $F(\vec e_1)=x$ and $F(\vec e_2)=y$ (the cocycle generated by $x$ and $y$).

(iii) If $F(\vec u) =(I-\mathbf R(\vec u))y$ is a coboundary for $\mathbf R$, then 
$(I-S)F(\vec e_1)=(I-T)F(\vec e_2)$ is a double coboundary.

(iv) If $z=(I-T)(I-S)h$ is a double coboundary, then the cocycle generated by 
 $x=(I-T)h$ and $y=(I-S)h$ is the coboundary $F(\vec u)=(I-\mathbf R(\vec u))h$.
\end{prop}
\begin{proof}
Since $\mathbf R(\vec e_1)=T$ and $\mathbf R(\vec e_2)=S$, (i) follows from
$$
F(\vec e_2) +SF(\vec e_1) =F(\vec e_1+\vec e_2) = F(\vec e_1)+TF(\vec e_2).
$$

We define (with empty sum defined as zero)
$$
F((n,m))= F(n\vec e_1+m\vec e_2) = \sum_{k=0}^{n-1}T^k x+T^n \sum_{j=0}^{m-1}S^j y.
$$
Some computations, using $(I-S)x=(I-T)y$, show that $F$ is a cocycle, and (ii) then follows.

Since a coboundary for $\mathbf R$ is a cocycle, (iii) follows from (i) and the definition of $F$.

Let $F$ be the cocycle generated by $x$ and $y$, given by (ii).  Then
$$
F(n\vec e_1+m\vec e_2) = \sum_{k=0}^{n-1}T^k x+T^n \sum_{j=0}^{m-1}S^j y=(I-T^n)h+T^n(I-S^m)h=
(I-T^nS^m)h,
$$
which proves (iv).
\end{proof}

\smallskip

As mentioned above, in this work we investigate the existence of joint coboundaries for the 
commuting $T$ and $S$, which are {\it not} double coboundaries. In view of Proposition \ref{yves},
the problem is, for actions of $\mathbb N^2$ or $\mathbb Z^2$, to find cocycles of the 
representation $\mathbf R$ (in different spaces) which are not coboundaries (non-triviality of
the first cohomology group). For example, if we have an ergodic action of a countable group on a 
measure space with an atom, then every cocycle is a (measurable) coboundary \cite[Exercise 2.9]{Sc}.
\smallskip

When we have an action of $\mathbb Z^2$ generated by commuting homeomorphisms $\theta$ and 
$\tau$ of a compact metric space $M$, it induces a representation $\bf R$ on $C(M)$, and any 
cocycle $F(\vec u) \in C(M)$ is a function on $M$. A special case of interest is that of an 
Anosov action on a differentiable manifold $M$ (see \cite{KaS}). In that case the study of 
cocyles and coboundaries is connected to rigidity proerties of the action. Katok and Spatzier 
\cite[Theorem 2.9]{KaS} proved that every $C^\infty$ (H\"older) cocycle  $F(\vec u)(t)$ 
of integral zero is a $C^\infty$ (H\"older) coboundary, and gave some applications.
Proposition \ref{yves} allows us to express this result in terms of joint and double 
coboundaries. The case of irrational rotations of the circle in Theorem \ref{irrational-pair} 
below shows that the analogue result of the Katok-Spatzier theorem need not hold for 
continuous cocycles of commuting (non-hyperbolic) diffeomorphisms.
\medskip

In Section 2 we study the existence of non-trivial double coboundaries in Banach spaces. 
We show that if $T \ne I$ and $S$ have the same fixed points, then there exist non-trivial 
double coboundaries.  If in addition $T$ and $S$ are mean ergodic, then the set of double 
coboundaries is closed if and only if both $T$ and $S$ are uniformly ergodic; if one of 
the operators is uniformly ergodic, then every joint coboundary is a double one.
\smallskip

In Section 3 we show that if $\theta$ and $\tau$ are commuting invertible measure preserving 
transformations of a standard probability space which generate an aperiodic $\mathbb Z^2$-action, 
then their induced unitary operators on $L_2$ have a joint coboundary which is not a double 
coboundary in $L_2$.
\smallskip

In Section 4 we study in detail pairs of irrational rotations of the unit circle $\mathbb T$,
with induced operators $T_\alpha$ and $T_\beta$ on different function spaces.  
We show the existence of a joint coboundary 
$\psi \in (I-T_\alpha)C(\mathbb T) \cap (I-T_\beta)C(\mathbb T)$ which is not even a measurable 
double coboundary -- there is no measurable $h$ such that $(I-T_\alpha)(I-T_\beta)h=\psi$.
\smallskip

In Section 5 we prove that when $T$ and $S$ are commuting mean ergodic contractions, then
$A_n(T)A_n(S)$ converges in operator norm if and only if $(I-T)X +(I-S)X$ is closed. We
prove that if $\theta$ and $\tau$ are  commuting ergodic measure preserving transformations
of a non-atomic probability space, and $U$ and $V$ are the isometries  they induce on $L_p$,
$1 \le p <\infty$, then $(I-U)L_p+(I-V)L_p$ is not closed.

\bigskip

\section{On double coboundaries of commuting contractions}

In this section we study the existence of non-trivial double coboundaries, and show that
when $X$ is reflexive, the operators have the same fixed points, and one of them
 is uniformly ergodic, then every joint coboundary is a double coboundary.
\smallskip

For a bounded operator $T$ on a Banach space $X$, we denote by $F(T)$ the space of fixed points
$\{x \in X: Tx=x\}$.
The following "ergodic decomposition" induced by commuting contractions was proved in
\cite[Theorem 2.4]{CL}.

\begin{theo}
Let $T_1,T_2,\dots ,T_d$ be commuting mean ergodic contractions of a Banach space $X$.
Then 
\begin{equation}\label{dec}
X = \overline{\sum_{1\le j \le d} F(T_j)} \oplus \overline{\prod_{1 \le j \le d}(I-T_j)X}.
\end{equation}
\end{theo}

\begin{prop} \label{homogeneous}
Let $T$ and $S$ be commuting contractions of a reflexive Banach space $X$. Then

(i) $(I-T)(I-S)z=0$ if and ony if $z \in \overline{F(T)+F(S)}$.

(ii) If $y \in (I-T)(I-S)X$, then there exists a unique $x \in \overline{(I-T)(I-S)X}$
with $(I-T)(I-S)x=y$.
\end{prop}
\begin{proof} (i) By commutativity and continuity, $z \in \overline{F(T)+F(S)}$
satisfies $(I-T)(I-S)z=0$. For the converse, let $(I-T)(I-S)z=0$. Denote 
$Y:=\overline{(I-T)(I-S)X}$, which is obviously $T$ and $S$ invariant. By the ergodic 
decomposition (\ref{dec}), $z=z_1+z_2$ with $z_1 \in \overline{F(T)+F(S)}$ and $z_2 \in Y$,
so $(I-T)(I-S)z_2=0$. Hence
$$
0= \frac1{N} \sum_{n=1}^N \sum_{k=0}^{n-1} \sum_{j=0}^{n-1} T^kS^j (I-T)(I-S)z_2 =
\frac1{N} \sum_{n=1}^N (I-T^n)(I-S^n)z_2 =
$$
$$
 z_2 + \frac1{N}  \sum_{n=1}^N (TS)^nz_2 -\frac1{N}\sum_{n=1}^N T^n z_2
 -\frac1{N}\sum_{n=1}^N S^n z_2.
$$
By reflexivity, the three averages converge. Each of the last two limits is in $F(T)+F(S)$
with $z_2 \in Y$, so (\ref{dec}) yields that each of these is zero.  The first limit is 
$TS$-invariant, so by the above $z_2$ is $TS$-invariant, which yields
$$
z_2 = \lim_{N\to \infty} \frac1{N} \sum_{n=1}^N (TS)^n z_2 = -z_2,
$$
 which proves $z_2=0$.
\smallskip

Proof of (ii): If $(I-T)(I-S)x_1 = (I-T)(I-S)x_2 $ for $x_1,x_2 \in Y$, then by (i)
$x_1-x_2 \in \overline{F(T)+F(S)}$, so $x_1-x_2=0$ by (\ref{dec}).
\end{proof}

{\bf Remark.} The construction of a solution $(I-T)(I-S)x=y$ given in \cite[Theorem 3.1]{CL}
yields the (unique) solution $x \in Y$,  by the invariance of $Y$ under $T$ and $S$.
\medskip

\begin{theo} \label{closures}
Let $T$ and $S$ be commuting contractions of a reflexive Banach space $X$. Then
$$
\overline{(I-T)(I-S)X} = \overline{(I-T)X\cap (I-S)X} = \overline{(I-T)X}\cap \overline{(I-S)X} .
$$
\end{theo}
\begin{proof}
Obviously 
$$
\overline{(I-T)(I-S)X} \subset \overline{(I-T)X\cap (I-S)X}  \subset
\overline{(I-T)X}\cap \overline{(I-S)X} .
$$
We show that 
$\overline{(I-T)(I-S)X} =  \overline{(I-T)X} \cap \overline{(I-S)X}$.
If not, by the inclusion above  there exists $z \in \overline{(I-T)X} \cap \overline{(I-S)X}$
which is not in $\overline{(I-T)(I-S)X} $; by the Hahn-Banach theorem  there exists 
$\phi \in X^*$ with $\phi(z) \ne 0$ and $\phi[(I-T)(I-S)X]=\{0\}$. We then have 
$(I-T^*)(I-S^*)\phi=0$, and since $X^*$ is reflexive, Proposition \ref{homogeneous}(i) yields 
that $\phi \in \overline{F(T^*)+F(S^*)}$.  But $\|\frac1n\sum_{k=1}^nT^k z\| \to 0$ and
$\|\frac1n\sum_{k=1}^n S^k z\| \to 0$, so for $\psi_1 \in F(T^*)$ and $\psi_2 \in F(S^*)$
we have
$$
(\psi_1+\psi_2)(z)= \frac1n\sum_{k=1}^n \psi_1(T^k z) + \frac1n\sum_{k=1}^n \psi_2(S^kz) \to 0.
$$ 
This yields that $\phi(z)=0$, a contradiction, which proves the theorem.
\end{proof}

Remember that the elements of $(I-T)(I-S)X$ are called {\it double coboundaries} (for $T$ 
and $S$). If $T$ or $S$ is the identity, then all double coboundaries are trivial (zero).
\smallskip

{\bf Example.} {\it $T$ and $S$ not the identity, with only trivial double coboundaries.}

\noindent 
Let $0\ne E \ne I$ be a (continuous linear) projection on a Banach space $X$, and put 
$T=E$, $S=I-E$; then $(I-T)(I-S)=0$. A sightly less trivial example is this:
Let $U$ be a unitary operator on a Hilbert space $H$, put 
$X =H\oplus H =\{(u,v): u,v \in H\}$, and define $T(u,v)=(Uu,v)$ and $S(u,v)=(u,Uv)$. 
Then $(I-T)(I-S)=0$.
\medskip

\begin{theo} \label{non-triv}
Let $T$ and $S$ be commuting contractions of a Banach space $X$. If $F(T) \subset F(S) \ne X$,
then there exist non-trivial double coboundaries.
\end{theo}
\begin{proof}
Assume all double coboundaries are zero, so for $x \in X$ we have $(I-T)(I-S)x=0$. Thus,
$(I-S)x \in F(T) \subset F(S)$, so $(I-S)x= \frac1n\sum_{k=1}^nS^k(I-S)x \to 0$, which yields
$Sx=x$. Since this is for any $x \in X$, $S=I$, a contradiciton.
\end{proof}

{\bf Examples.} 1. Let $T$ and $S$ be induced by commuting ergodic probability preserving 
transformations, e.g. irrational rotations of the unit circle. Then the theorem applies.

2. let $\mu$ and $\nu \ne \delta_1$ be probabilities on the unit circle $\mathbb T$  
such that the closed subgroup of $\mathbb T$ generated by the support of $\mu$ is all of $G$, 
and put $Tf=\mu*f$ and $Sf=\nu*f$ on $L_p$. The condition on $\mu$ implies $F(T)=\{constants\}$,
so the theorem applies.
\medskip

\begin{theo} \label{closed}
Let $T$ and $S$ be commuting mean ergodic contractions on a Banach space $X$, with $F(T)=F(S)$.
Then $(I-T)(I-S)X$ is closed if and only if both $T$ and $S$ are uniformly ergodic.
\end{theo}
\begin{proof} 
By \cite[Remark 2.5]{CL}; see (\ref{dec}),  the assumption $F(T)=F(S)$ implies 
\begin{equation} \label{closure}
\overline{(I-S)X}=\overline{(I-T)X}=\overline{(I-T)(I-S)X}.
\end{equation}

(a) If $(I-T)(I-S)X$ is closed, then $(I-T)X$ and $(I-S)X$ are closed, and by \cite{L} $T$
and $S$ are uniformly ergodic.

(b) Denote $Y:= \overline{(I-T)(I-S)X}$. If $T$ and $S$ are uniformly ergodic, then
$(I-T)X=(I-S)X=Y$ and $I-T$ and $I-S$ are invertible on $Y$. Let $y \in Y$. Then there is 
$x \in Y$ with $(I-T)x=y$, and then a $z \in Y$ with $(I-S)z=x$. Hence $(I-T)(I-S)z=y$.
Thus $(I-T)(I-S)X$ is closed.
\end{proof}

{\bf Example.} Let $T$ and $S$ be induced on $L_2$ by commuting ergodic probability preserving 
transformations on a Lebesgue space. Then by \cite{IT} $\sigma(T)=\sigma(S) =\mathbb T$,
so neither is uniformly ergodic; hence by the theorem $(I-T)(I-S)L_2$ is not closed.
\medskip

Let $T$ and $S$ be contractions of a Banach space $X$. The elements of $(I-T)X \cap (I-S)X$
are called {\it joint (or common) coboundaries} (of $T$ and $S$). When $T$ and $S$ commute,
double coboundaries are joint coboundaries, and Theorem \ref{non-triv} yields existence
of non-trivial joint coboundaries. Theorem \ref{closures} shows that in reflexive spaces,
every joint coboundary can be approximated by double coboundaries. We want to address 
the question of existence of joint coboundaries which are {\it not} double coboundaries.
Adams and Rosenblatt \cite{AR} studied existence of non-trivial joint coboundaries in the 
non-commutative case. 
\smallskip

\begin{cor} \label{equivalences}
Let $T$ and $S$ be commuting mean ergodic contractions on a Banach space $X$, with $F(T)=F(S)$.
Then the following are equivalent:

(i) The space of double coboundaries $(I-T)(I-S)X$ is closed.

(ii) The space of joint coboundaries $(I-T)X\cap (I-S)X$ is closed.

(iii) Both $T$ and $S$ are uniformly ergodic.
\end{cor}
\begin{proof} By (\ref{closure}), (i) implies 
$$
\overline{(I-T)X}=(I-T)(I-S)X \subset (I-T)X \cap (I-S)X \subset \overline{(I-T)X},
$$
which yields (ii).

By (\ref{closure}), and (ii), we have
$$
\overline{(I-T)X} = (I-T)X \cap (I-S)X \subset (I-T)X,
$$
which yields that $(I-T)X$ is closed, so $T$ is uniformly ergodic, and similarly for $S$.
\end{proof}

\begin{lem} \label{same}
Let $T$ and $S$ be commuting mean ergodic contractions on a Banach space $X$, with $F(T)=F(S)$,
and assume that $T$ is uniformly ergodic. Then:

(i)  Every coboundary of $S$, in particular every joint coboundary, is a double coboundary.

(ii) $\displaystyle{\lim_{\min(n,m) \to \infty} \|\frac1{nm}\sum_{k=0}^n\sum_{j=0}^m T^k S^j -E\|
\to 0}$, where $E$ is the projection on $F(T)$ with null space $(I-T)X$.
\end{lem}
\begin{proof} 
The assumption $F(T)=F(S)$ implies $\overline{(I-S)X}=\overline{(I-T)X}= \overline{(I-T)(I-S)X}$,
by \cite[Remark 2.5]{CL}.  

(i) Let $x=(I-S)z$ be a coboundary of $S$. By mean ergodicity of $S$, we can take (uniquely) 
$z \in \overline{(I-S)X}$. By uniform ergodicity of $T$, $Y:=(I-T)X$ is closed (and $T$ 
invariant), and $I-T$ is invertible on $Y$. Then by the above $z \in (I-T)X$, which yields that
$$
x=(I-S)z =(I-S)(I-T)(I-T_{|Y})^{-1}z
$$
is a double coboundary.

(ii) Put $M_n(T) = \frac1n\sum_{k=0}^n T^k$, and $M_m(S) =\frac1m\sum_{j=0}^m S^j$. 
By uniform ergodicity, $\|M_n(T)-E\| \to 0$. Since $T$ and $S$ commute and have the same ergodic
decomposition, $SE=E$. Hence
$$
\|M_m(S)M_n(T)-E\|=\|M_m(S)(M_n(T)-E)\| \le \|M_n(T)-E\| \to 0 
\quad \text{as } \min(n,m) \to \infty.
$$
\end{proof}
 
{\bf Remark.} The unitary operator $T$ induced on the complex $L_2$ by an ergodic invertible 
measure preserving transformation of a non-atomic probability space (with $T^k \ne I$ for 
$k \in \mathbb N$) is not uniformly ergodic, since its spectrum is $\mathbb T$ \cite{IT} 
(this is immediate for an irrational rotation of the unit circle, since the eigenvalues 
are dense in $\mathbb T$). 


\begin{theo} \label{UE}
Let $T$ and $S$ be commuting mean ergodic contractions on a Banach space $X$ with $F(T)=F(S)$.
Then $T$ is uniformly ergodic if and only if every coboundary of $S$ is a double coboundary.
\end{theo}
\begin{proof} If $T$ is uniformly ergodic, we apply Lemma \ref{same}(i). 

Assume that every coboundary of $S$ is a double coboundary.
Fix $z \in \overline{(I-T)X}=\overline{(I-S)X}$ and put $x=(I-S)z$. By assumption, there 
is  $y \in X$ with $x=(I-T)(I-S)y$. Hence $(I-S)[z-(I-T)y]=0$, so $z -(I-T)y \in F(T)$, 
which yields $z=(I-T)y$, since $z \in \overline{(I-T)X}$. Thus $(I-T)X$ is closed, 
so $T$ is uniformly ergodic by \cite{L}.
\end{proof}

{\bf Remark.} Lemma \ref{same}(i) yields that when $F(T)=F(S)$, a necessary condition for the 
existence of a joint coboundary which is not a double coboundary is that neither $T$ nor $S$
be uniformly ergodic.

\begin{cor} \label{joint1}
Let $T$ and $S$ be commuting mean ergodic contractions on a Banach space $X$ with $F(T)=F(S)$, 
and assume $T$ is not uniformly ergodic. If $(I-S)X \subset (I-T)X$, then there exists a joint 
coboundary which is not a double coboundary.
\end{cor}
\begin{proof}
Since $T$ is not uniformly ergodic, 
by Theorem \ref{UE} there exists $y \in (I-S)X$ which is not a double coboundary. Then $y$
is a joint coboundary, since $(I-S)X \subset (I-T)X$. 
\end{proof}


{\bf Remark.}  If $T$ and $S$ are induced by invertible ergodic probability preserving 
transformations, then by Kornfeld \cite[Theorem 2]{Ko}, the assumption $(I-S)X \subset (I-T)X$ 
implies that $S=T^k$ for some $k \in \mathbb Z$. Therefore, if $T \ne S^k$ and $S \ne T^k$, 
then there are coboundaries of $T$ and of $S$ which are not joint coboundaries.

\begin{theo} \label{powers}
Let $R$ be a mean ergodic contraction, which is not uniformly ergodic,  on a Banach space $X$,
and let $T=R^k$ and $S=R^j$ be mean ergodic (e.g. $X$ is reflexive), with $F(T)=F(S)=F(R)$. 
Then $T$ and $S$ have a joint coboundary which is not a double coboundary.
\end{theo}
\begin{proof} 
Since $R$ is not uniformly ergodic, neither are $T$ nor $S$. By Theorem \ref{joint1}, there is 
a joint coboundary $u$ for $R$ and $S$ which is not a double coboundary for them. We put
$u=(I-R)x=(I-S)y$, and define $v=(I-T)x$. Then
$$
v=(I-T)x=(I-R^k)x= \sum_{n=0}^{k-1}R^n(I-R)x =\sum_{n=0}^{k-1}R^k(I-S)y = \sum_{n=0}^{k-1}R^ku,
$$
so $v$ is a joint coboundary for $T$ and $S$. We show it is not a double coboundary of 
$T$ and $S$. Assume $v=(I-T)(I-S)z$. Then 
$$
\sum_{n=0}^{k-1}R^ku = v=\sum_{n=0}^{k-1} R^k(I-R)(I-S)z,
$$
which yields $(I-R^k)[u-(I-R)(I-S)z]=0$, so $u-(I-R)(I-S)z \in F(T)=F(R)$, so we obtain
$$
(I-R)x= u=[u-(I-R)(I-S)z]+ (I-R)(I-S)z.
$$
Uniqueness in the ergodic decomposition (with respect to $R$) yields $u=(I-R)(I-S)z$, which
means that $u$ is a double coboundary of $R$ and $S$, contradicting the choice of $u$.
\end{proof}

{\bf Remark.} If $T$ is a mean ergodic contraction, taking $S=T$ we obtain that  $T$
is uniformly ergodic if and only if $(I-T)X=(I-T)^2X$.

\section{Joint coboundaries of commuting measure-preserving transformations}

In this section we show that for  commuting invertible measure-preserving transformations 
$\theta$ and $\tau$ of a standard probability space $(\Omega, \B, \P)$ which generate an
aperiodic $\mathbb Z^2$ action, their induced unitary operators on $L_2(\Omega,\P)$ 
have a joint coboundary in $L_2$ which is not a double coboundary in $L_2 $. 
\smallskip

A bounded operator $T$ on a Banach space is called {\it aperiodic} if $T^n \ne I$ for any 
$n \in\mathbb N$  (see \cite{Fo}).  If $T^k=I$, then $T$ is power-bounded and uniformly ergodic.
Hence for any commuting mean ergodic contractions with $F(T)=F(S)$, to have joint 
coboundaries which are not double coboundaries, it is necessary that both operators be 
aperiodic (by Lemma \ref{same}).
A probability preserving transformation $\theta$ is called {\it aperiodic} if $\theta^n \ne id$ 
for any $n\ge 1$, i.e. its induced operator on $L_2$ is aperiodic
(a more restrictive definition is given in \cite{IT}). 
Ergodic probability preserving transformations of a standard probability space are aperiodic.

If $\theta$ and $\tau$ are invertible probability preserving transformations on $(\Omega,\P)$, 
we say that the $\mathbb Z^2$-action they generate is {\it aperiodic} (see \cite{Co}, \cite{KW}) 
if for $j,k \in\mathbb Z$ which are not both zero, $\P(\{x \in \Omega: \theta^j\tau^kx=x\}) =0$;
in that case, the induced unitary operators $U$ and $V$ satisfy $U^jV^k \ne I$ whenever
$j$ and $k$ are not both zero.
\medskip

Let $U$ and $V$ be commuting unitary operators on a complex Hilbert space $H$. They generate
a unitary representation of $\mathbb Z^2$, to which we apply the general Stone spectral theorem 
(e.g. \cite{Am}) to obtain:
 {\it There exists a (unique) projection valued spectral measure $E(\cdot)$ on the Borel sets
of $\mathbb T^2 = \widehat{\mathbb Z^2}$ such that (in the strong operator topology)}
$$
U^nV^m = \int_{\mathbb T^2} z_1^nz_2^m dE(z_1,z_2) \qquad n,m \in \mathbb Z.
$$
Hence $P(U,V)= \int_{\mathbb T^2} P(z_1,z_2) dE(z_1,z_2)$ for every polynomial $P$ in two
commuting variables.
We denote by $\sigma_f(\cdot):=\langle E(\cdot)f,f \rangle$ the spectral measure of $f \in H$,
and obtain that $\|P(U,V)f\|^2 = \int_{\mathbb T^2} |P(z_1,z_2)|^2 d\sigma_f(z_1,z_2)$.
\smallskip

The spectral measures of $U$ and $V$ are obtained from $E$ by $E_U(B)=E(B\times \mathbb T)$
and $E_V(B)=E(\mathbb T \times B)$ for Borel subsets of $\mathbb T$ (see \cite[Theorem 3]{Am}).

\begin{prop} \label{spectral-cob}
Let $U$ and $V$ be commuting unitary operators on a complex Hilbert space $H$, and $f \in H$.

(i) $f \in (I-U)H$ if and only if 
$\displaystyle{ \int_{\mathbb T^2} \frac{d\sigma_f(z_1,z_2)}{|z_1-1|^2} < \infty.}$
\smallskip

(ii) $f \in (I-U)(I-V)H$ if and only if
$\displaystyle{ \int_{\mathbb T^2} \frac{d\sigma_f(z_1,z_2)}{|z_1-1|^2|z_2-1|^2} < \infty.}$
\end{prop}
\begin{proof} (i) is a well-known consequence of Browder's theorem  \cite{Br}.
The details of the proof will be clear from the proof of (ii) below.
\smallskip

It was proved in \cite[Theorem 3.1]{CL} that $ f \in (I-U)(I-V)H$ if and only if
$$
\sup_n \|\sum_{k=0}^{n-1}\sum_{j=0}^{n-1}U^kV^jf\| < \infty.
$$
Assume convergence of the integral in (ii); it implies that 
$\sigma_f(\{1\}\times\mathbb T)= \sigma_f(\mathbb T\times\{1\})=0$, so
$$ 
\|\sum_{k=0}^{n-1}\sum_{j=0}^{n-1}U^kV^jf\|^2 = 
\int_{\mathbb T^2} |\sum_{k=0}^{n-1}\sum_{j=0}^{n-1}z_1^kz_2^j |^2 d\sigma_f(z_1,z_2) =
\int_{\mathbb T^2} \frac{|z_1^n-1|^2|z_2^n-1|^2}{|z_1-1|^2|z_2-1|^2} d\sigma_f(z_1,z_2) .
$$
This yields 
$$\sup_n \|\sum_{k=0}^{n-1}\sum_{j=0}^{n-1}U^kV^jf\|^2  \le
16 \int_{\mathbb T^2} \frac{d\sigma_f(z_1,z_2)}{|z_1-1|^2|z_2-1|^2} < \infty,
$$ 
which proves that $f \in (I-U)(I-V)H$.
\smallskip

Assume now that $f \in (I-U)(I-V)H$. Then
$$
\sup_N \big\| \frac1{N^2} \sum_{n=1}^N\sum_{m=1}^N
\sum_{k=0}^{m-1}\sum_{j=0}^{n-1}U^kV^jf \big\|^2 =M < \infty.
$$
Since the spectral measure of $U$ is $E_U(B)=E(B\times \mathbb T)$, we have 
$E(\{1\}\times\mathbb T)g=E_U(\{1\})g=0$ when $g \in (I-U)H$. Hence $f \in (I-U)(I-V)H$ yields
$\sigma_f(\{1\}\times \mathbb T)=\sigma_f(\mathbb T \times\{1\}) =0$, so in particular
$\sigma_f(\{(1,1)\})=0$.
$$
\big\| \frac1{N^2} \sum_{n=1}^N\sum_{m=1}^N \sum_{k=0}^{m-1}\sum_{j=0}^{n-1}U^kV^jf \big\|^2 =
\int_{\mathbb T^2} 
\big| \frac1{N^2} \sum_{n=1}^N\sum_{m=1}^N \sum_{k=0}^{m-1}\sum_{j=0}^{n-1}z_1^kz_2^j \big|^2 
d\sigma_f(z_1,z_2)=
$$
$$
\int_{\mathbb T^2} 
\Big|\Big(\frac1N\sum_{m=1}^N (z_1^m-1)\Big)\Big(\frac1N \sum_{n=1}^N(z_2^n-1)\Big)\Big|^2\
\frac1{|z_1-1|^2|z_2-1|^2}  d\sigma_f(z_1,z_2).
$$
The limit $\lim_{N\to \infty} 
\Big|\big(\frac1N\sum_{m=1}^N (z_1^m-1)\big)\big(\frac1N \sum_{n=1}^N(z_2^n-1)\big)\Big|^2$ exists
for every $(z_1,z_2) \in \mathbb T$; it is $0$ when $z_1=1$ or $z_2=1$, 
and 1 otherwise; hence the limit is 1 $\sigma_f$ a.e.  By Fatou's lemma we have
$$
\int_{\mathbb T^2} \frac{d\sigma_f(z_1,z_2)}{|z_1-1|^2|z_2-1|^2} =
$$
$$
 \int_{\mathbb T^2} \lim_{N\to \infty} 
\Big|\Big(\frac1N\sum_{m=1}^N (z_1^m-1)\Big)\Big(\frac1N \sum_{n=1}^N(z_2^n-1)\Big)\Big|^2\
\frac1{|z_1-1|^2|z_2-1|^2}  d\sigma_f(z_1,z_2)
\le
$$
$$
\liminf_{N\to \infty}  \int_{\mathbb T^2} 
\Big|\Big(\frac1N\sum_{m=1}^N (z_1^m-1)\Big)\Big(\frac1N \sum_{n=1}^N(z_2^n-1)\Big)\Big|^2\
\frac1{|z_1-1|^2|z_2-1|^2}  d\sigma_f(z_1,z_2)
\le M,
$$
which proves (ii).
\end{proof}

It follows from Proposition \ref{spectral-cob}(i) that  $f$ is a joint coboundary if and only if
\begin{equation} \label{spectral-joint}
\int_{\mathbb T^2} \frac{|z_1-1|^2 +|z_2-1|^2}{|z_1-1|^2|z_2-1|^2} d\sigma_f(z_1,z_2) <\infty.
\end{equation}
The problem of finding a joint coboundary for $U$ and $V$ which is not a double coboundary is 
therefore the problem of finding $0 \ne f \in H$ such that  $\sigma_f$ satisfies
(\ref{spectral-joint}) and the integral in Proposition \ref{spectral-cob}(ii) diverges. 
This requires a deeper study of spectral measures, summarized below.
\medskip

Abstract (orthogonal projection valued) spectral measures (called spectral families 
in \cite{Am}), defined on a measurable space $(S,\Sigma)$ (and not necessarily connected 
to any unitary representation or any operator), were studied in the books of Halmos 
\cite{Hal} and Nadkarni \cite{Na}.
We fix a complex Hilbert space $H$ and a  spectral measure $E(\cdot)$ with values in 
$B(H)$, and define as before the spectral measure of $f \in H$ by 
$\sigma_f(\cdot):=\langle E(\cdot)f,f \rangle$, which is a positive finite measure. 
For $f \in H$ we define the {\it cyclic subspace}  $Z(f)$ generated by $f$ as the 
closed linear manifold generated by $\{E(A)f: A \in \Sigma\}$.  The orthogonal projection 
on $Z(f)$  commutes with $E(\cdot)$ \cite[p. 91]{Hal}; hence also $Z(f)^\perp$ is invariant 
under $E(\cdot)$, and $Z(g) \perp Z(f)$ for $g\perp Z(f)$.

\begin{theo} \label{maximal-spectral}
Let $H$ be a separable complex Hilbert space. Then there exists a vector $\psi \in H$ such that 
$\sigma_\psi(A)=0$ if and only if $E(A)=0$; hence $\sigma_g << \sigma_\psi$ for every $g \in H$.
\end{theo}
The proof is given in \cite[p. 11-12]{Na}. The equivalence class of $\sigma_\psi$ is called
the {\it maximal spectral type} of $E$, and is often denoted by $\sigma_\psi$
instead of $[\sigma_\psi]$. More detailed information is given by the Hahn-Hellinger theorem
\cite{Na}.

%
%
%

\begin{theo} \label{existence}
{\rm \cite[p. 104]{Hal}}.
Let $\nu$ be a finite measure on $(S,\Sigma)$. If $\nu << \sigma_f$ for some $f \in H$,
then there exists $ g \in Z(f)$ such that $\nu=\sigma_g$. When $\nu \sim \sigma_f$ we
have $Z(g)=Z(f)$.
\end{theo}

%
%
%
%
%
%

The following proposition and its proof are inspired by the one-dimensional result for unitary
operators in \cite[p. 290]{IT}. 
We use the Harte {\it joint spectrum} of $d$ commuting operators in a complex Banach space $X$,
with respect to the Banach algebra $B(X)$ of all bounded linear operators on $X$ \cite{Ha1}.
For completeness we repeat the definition.
\smallskip

{\bf Definition.} The {\it joint spectrum} of the operators $T_1,\dots,T_d \in B(X)$,
denoted by $\sigma(T_1,\dots,T_d)$ or $\sigma(\{T_j\})$, is the set of 
$(\lambda_1,\dots,\lambda_d) \in \mathbb C^d$ such that one of the equations
\begin{equation} \label{left}
\sum_{j=1}^d A_j(\lambda_j I - T_j)=I \qquad A_1,\dots,A_d \in B(X)
\end{equation}
\begin{equation} \label{right}
\sum_{j=1}^d (\lambda_j I - T_j)B_j=I \qquad B_1,\dots,B_d \in B(X)
\end{equation}
has no solution. The joint spectrum is closed \cite[p. 872]{Ha1}, and the following 
properties follow directly from the definition:
\begin{equation} \label{product}
\sigma(T_1,\dots,T_d) \subset \sigma(T_1) \times \dots \times \sigma(T_d).
\end{equation}
\begin{equation} \label{ap}
\inf_{\|x\|=1} \sum_{j=1}^d \|(\lambda_j I-T_j)x\| =0\quad \Longrightarrow \quad
(\lambda_1,\dots,\lambda_d) \in \sigma(T_1,\dots,T_d),
\end{equation}
Since equation (\ref{left}) cannot have a solution. The set of points satifying (\ref{ap}) 
is the {\it approximate point spectrum} $\sigma_\pi(T_1,\dots,T_d)$.
\smallskip

{\it If $X= X_1\oplus X_2$  with $X_1$ and $X_2$ each invariant under all the $T_j$,
then the restrictions $T_j^{(k)}$ of $T_j$ to $X_k$ satisfy}
\begin{equation} \label{direct}
\sigma(T_1,\dots,T_d) \subset 
\sigma(T_1^{(1)},\dots,T_d^{(1)}) \cup \sigma(T_1^{(2)},\dots,T_d^{(2)}).
\end{equation}

For additional information when the operators act in a Hilbert space, see \cite[Chapter 5]{LKMV}. 
Note that by \cite[Proposition 2.10]{Va}, \cite[p. 30]{BS}, the joint spectrum of operators 
on $H$ is a subset of the Taylor spectrum \cite{Ta} $\sigma_{\mathbf T}(T_1, \dots,T_d)$, 
while for normal operators on $H$, these spectra are equal, by \cite[pp. 30-31]{MPR}.

\begin{prop} \label{torus}
Let $U$ and $V$ be commuting isometries on a complex Hilbert space $H$.
If for every $n>2$ there exists a vector $v_n \ne 0$ such that the vectors
$\{U^kV^j v_n: 0\le k \le n,\ 0\le j \le n\}$ are orthogonal, then 
$\mathbb T\times \mathbb T \subset \sigma(U,V).$ 

If $U$ and $V$ are invertible (unitary operators), then
$\mathbb T\times \mathbb T =\sigma(U,V).$ 
\end{prop}
\begin{proof}
We fix $n>2$ and denote $v=v_n$. Let $|\lambda|=|\nu| =1$, and define 
$$
x_n := \sum_{k=0}^{n-1}\sum_{j=0}^{n-1} \lambda^{n-k}U^k \nu^{n-j}V^j v.
$$
By the assumed orthogonality, $\|x_n\|^2=n^2\|v\|^2 \ne 0$. We compute
$$(\lambda I -U)x_n =
\sum_{j=0}^{n-1}\nu^{n-j}V^j [(\lambda I- U)\sum_{k=0}^{n-1} \lambda^{n-k}U^k v]=
$$
$$
\sum_{j=0}^{n-1}\nu^{n-j}V^j [\sum_{k=0}^{n-1} (\lambda^{n+1-k}U^k v - \lambda^{n-k}U^{k+1} v)]=
\sum_{j=0}^{n-1}\nu^{n-j}V^j (\lambda^{n+1}v - \lambda U^nv).
$$
By the orthogonality assumption we obtain $\|(\lambda I-U)x_n\|^2 =2n \|v\|^2 =\frac2n \|x_n\|^2$,
and similarly $\|(\nu I-V)x_n\|^2 =\frac2n \|x_n\|^2$, so
\begin{equation} \label{approximate}
\|(\lambda I-U)x_n\| +\|(\nu I-V)x_n\| = 2\frac{\sqrt 2}{\sqrt n} \|x_n\|. 
\end{equation}
Thus, (\ref{approximate}) shows that $ (\lambda, \nu)$ is in the "approximate point spectrum"
$\sigma_\pi(U,V)$. 

Hence $\mathbb T\times \mathbb T \subset \sigma_\pi(U,V) \subset \sigma(U,V)$, 
the last inclusion by (\ref{ap}).
\smallskip

When $U$ and $V$ are unitary, 
$\sigma(U,V) \subset \sigma(U) \times \sigma(V) \subset \mathbb T \times \mathbb T$ 
by (\ref{product}), which completes the proof.
\end{proof}

\begin{cor} \label{alex2}
Let $\theta$ and $\tau$ be commuting measure preserving transformations of a 
non-atomic probability space $(\Omega,\P)$ which generate a free $\mathbb N_0^2$ action.
Then for $1\le p <\infty$,  the joint spectrum of the isometric operators $U$ and $V$ 
induced on the complex $L_p$  by $\theta$ and $\tau$ contains $\mathbb T \times \mathbb T$.
If $\theta$ and $\tau$ are invertible, then $\sigma(U,V) =\mathbb T \times \mathbb T$.
\end{cor}
\begin{proof}
The Rokhlin-Kakutani lemma for free $\mathbb N_0^d$ actions, proved by Avila and Candela 
\cite{AC} (for non-periodic $\mathbb Z^d$ actions see Conze \cite[Lemme 3.1]{Co} or Katznelson 
and Weiss \cite[Theorem 1]{KW}), yields that for every $n$ there is a measurable set $E_n$
such that the sets $\{\theta^{-k}\tau^{-j}E_n: 0\le k \le n,\ 0\le j \le n\}$ are disjoint.
For $p=2$ the vector $v_n=1_{E_n}$ satisfies the assumptions of Proposition \ref{torus}.

When $p \ne 2$, note that for $v_n=1_{E_n}$ the functions 
$\{U^kV^j v_n: 0\le k \le n,\ 0\le j \le n\}$ have disjoint supports. Replacing 
orthogonality by disjointness of supports and taking norms to power $p$ instead of squares 
in the proof of  Proposition \ref{torus}, that proof yields that also in $L_p$ we have 
$\mathbb T\times \mathbb T \subset \sigma(U,V)$, with equality when $\theta$ and $\tau$ 
are invertible.
\end{proof}

{\bf Remarks.} 1. The one-dimensional result for invertible transformations and $p=2$ was obtained
in \cite[Corollary 1]{IT}. It could have been proved there (even for $p\ne 2$) by using the 
original Rokhlin's lemma and Remark 5 in \cite[p. 290]{IT}. Another proof  (for $p=2$)
was given in \cite[Corollaire 1]{Fo}.

2. If $\theta$ and $\tau$ are invertible and  generate an aperiodic $\mathbb Z^2$-action,
then the $\mathbb N_0^2$-action is free.
\medskip

We now return to the problem of joint coboundaries which are not double coboundaries.


\begin{theo} \label{mpt}
Let $\theta$ and $\tau$ be commuting invertible measure preserving transformations of a 
standard probability space $(\Omega, \mathbf B, \P)$  which generate an aperiodic $\mathbb Z^2$
action, and let $U$ and $V$ be their coresponding unitary operators on $L_2(\Omega,\mathbf B,\P)$.
Then there exists a function $g \in (I-U)L_2\cap (I-V)L_2$ which is not in $(I-U)(I-V)L_2$.
\end{theo}
\begin{proof} Let $E(\cdot)$ be the spectral measure of the pair $(U,V)$ given by the general 
Stone spectral theorem. Since  (the complex) $L_2(\P)$ is separable (standard probability space), 
by Theorem \ref{maximal-spectral}, there exists $\psi \in H$ such that $\sigma_\psi$ is the maximal 
spectral type  of $E$.

For $(1,1) \ne (z_1,z_2) \in \mathbb T^2$ define
$$
\phi(z_1,z_2) :=  \frac{|z_1-1|^2  |z_2-1|^2}{|z_1-1|^2 + |z_2-1|^2},
$$ 
and put $\phi(1,1)=0$. Then $\phi(z_1,z_2) \le |z_2-1|^2 \le 4$, so $\phi $ is bounded, 
and vanishes on
$A:= \{(z_1,z_2): \phi(z_1,z_2)=0\} = (\{1\}\times \mathbb T)\cup(\mathbb T\times \{1\})$.
Since $U\ne I$ and $V \ne I$, $E(\{1\}\times \mathbb T)=E_U(\{1\})\ne I$ and 
$E(\mathbb T \times \{1\})=E_V(\{1\})\ne I$; hence $\sigma_\psi(A^c):=\sigma_\psi(\mathbb T^2-A)>0$.

For any finite measure $\mu << \sigma_\psi$  
 we define $\nu_\mu$ by $d\nu_\mu/d\mu :=\phi$. Then $\nu_\mu$ is a finite measure with
$\nu_\mu(A)=0$, $\nu_\mu \ne 0$ when $\mu(A^c) > 0$, and $\nu_\mu$ satisfies
$$
\int_{\mathbb T^2} \frac{|z_1-1|^2 +|z_2-1|^2}{|z_1-1|^2|z_2-1|^2} d\nu_\mu =
\int_{A^c} \frac{|z_1-1|^2 +|z_2-1|^2}{|z_1-1|^2|z_2-1|^2} \phi(z_1,z_2)d\mu
= \mu(A^c) < \infty.
$$
Hence $\nu_\mu$ satisfies (\ref{spectral-joint}). We now show that for some $\mu$ we have
$\displaystyle{ \int_{\mathbb T^2} \frac{d\nu_\mu(z_1,z_2)}{|z_1-1|^2|z_2-1|^2} = \infty.}$
Assume not; then for every $\mu << \sigma_\psi$ with $\mu(A^c) >0$ we have
$$
\int_{A^c} \frac{d\mu}{|z_1-1|^2+|z_2-1|^2} =
\int_{A^c} \frac{\phi(z_1,z_2)d\mu}{|z_1-1|^2|z_2-1|^2} =
 \int_{\mathbb T^2} \frac{d\nu_\mu(z_1,z_2)}{|z_1-1|^2|z_2-1|^2}  < \infty.
$$
By a well-known lemma (Lemma \ref{bounded} below) we conclude that 
$(|z_1-1|^2 + |z_2-1|^2)^{-1} \in L_\infty(A^c,\sigma_\psi)$. 

By a theorem of Hastings \cite[Theorem 3]{Has}, the support of $E(\cdot)$ (which is 
the support of $\sigma_\psi$), equals the joint spectrum $\sigma''(U,V)$ with respect to the 
double commutant of $(U,V)$; but this equals the Harte spectrum $\sigma(U,V)$ by 
\cite[Proposition 4 and Example (A)]{MPR}. See also \cite[Theorem 2.2]{Pa}.
By our aperiodicity assumption on the transformations, Corollary \ref{alex2} yields
$\sigma(U,V) = \mathbb T^2$, so the support of $\sigma_\psi$ is all of $\mathbb T^2$.
Hence $(|z_1-1|^2 + |z_2-1|^2)^{-1}$, which is unbounded in any neighborhood of $(1,1)$, 
cannot be in $L_\infty(A^c,\sigma_\psi )$. This contradiction shows that for some 
$\mu_0 << \sigma_\psi$ (with $\mu_0(A^c)>0$), we have 
$\displaystyle{ \int_{\mathbb T^2} \frac{d\nu_{\mu_0}(z_1,z_2)}{|z_1-1|^2|z_2-1|^2} = \infty.}$
By the construction, $0\ne \nu_{\mu_0} << \sigma_\psi$, and by  Theorem \ref{existence}
there exists $g \in Z(f) \subset L_2$ such that $\sigma_g= \nu_{\mu_0}$. By what we saw, 
$\sigma_g$ satisfies (\ref{spectral-joint}), and the integral with respect to $\sigma_g$ in 
Proposition \ref{spectral-cob}(ii) diverges. Hence $g$ is a joint coboundary which is not a 
double coboundary. It can be shown (using the spectral theorem, \cite[Theorem 1, p. 95]{Hal}
and  $L_2$-density of trigonometric polynomials on $\mathbb T^2$)
that $Z(\psi)$ is the closed linear manifold generated by the $\mathbb Z^2$-orbit 
$\{U^nV^m \psi: m,n \in \mathbb Z\}$. 

The above proof yields $g$ in the {\it complex} $L_2$. Let $L_2^{(\mathbb R)}$ be the real $L_2$,
which is invariant under $U$ and $V$. If $g=(I-U)h$, then $\Re g=(I-U)\Re h$, etc. Hence
$\Re g$ and $\Im g$ are both joint coboundaries. If there are $h_1,h_2 \in L_2^{(\mathbb R)}$
such that $(I-U)(I-V)h_1 =\Re g$ and $(I-U)(I-V)h_2= \Im g$, then $(I-U)(I-V)(h_1+ih_2) =g$,
contradicting the choice of $g$. Hence $\Re g$ or $\Im g$ is a joint coboundary in 
$L_2^{(\mathbb R)}$ which is not a double coboundary.
\end{proof}

%

\begin{lem} \label{bounded}
Let $h \ge 0$ be a measurable function on a finite measure space $(S, \Sigma, \sigma)$.
If $\int hf\,d\sigma < \infty$ for every $0 \le  f \in L_1(\sigma)$, 
then $h \in L_\infty(\sigma)$.
\end{lem}
\begin{proof} If $h$ is not in $L_\infty$, for every $n$ there exists $f_n \in L_1$ with 
$\|f_n\|_1=1$ and $|\int hf_n \,d\sigma| > 2^n$. Then $\int h|f_n|\,d\sigma > 2^n$, and 
$f =\sum_{n=1}^\infty  2^{-n}|f_n|$ satisfies $\|f\|_1=1$. But $\int hf\,d\sigma =
\sum_{n=1}^\infty 2^n \int h|f_n|d\sigma =\infty$, a contradiction.
\end{proof}

%
%

{\bf Remark.} The condition that the $Z^2$-action be free is not necessary. 
If $\theta$ is aperiodic on a standard probability space, then $\sigma(U) =\mathbb T$,
so $U$ is not uniformly ergodic, and taking $\tau=\theta$ the theorem holds, since 
$(I-U)H \ne (I-U)^2H$, by the remark to Theorem \ref{powers}.
\medskip

{\bf Example.} {\it Commuting transformations with a joint non-double coboundary in any $L_p$}

The measure space is $\mathbb N^2$ with the counting measure $m$. The transformations are
$\theta(j,k) =(j+1,k)$ and $\tau(j,k)=(j,k+1)$ for $j\ge 1,\  k\ge 1$. The counting measure is 
not invariant, but subinvariant for $\theta$ and $\tau$. Fix $1 \le p<\infty$, and put 
$X:= \ell_p(\mathbb N^2) := \{(f_{j,k}): \sum_{j,k =1}^\infty  |f_{j,k}|^p < \infty\}$. 
Define  $Uf =f\circ \theta$ and $Vf=f\circ\tau$, so $(Uf)_{j,k} =f_{j+1,k}$ and
$(Vf)_{j,k}= f_{j,k+1}$. Then $U$ and $V$ are commuting contractions of $X$,
satisfying $U^n \to 0$ and $V^n \to 0$ strongly.

Fix $p>1$ and put $a=p+1$. Define $h \in X$ by $h_{j,k} =1/(j+k)^{a/p}$ for $j,k \ge 1$. Then $h \in X$.
By definition $Uh=Vh= \big(1/(j+k+1)^{a/p}\big)$, so $f:=(I-U)h$ is  a joint coboundary.
Suppose $f=(I-U)(I-V)q$ for some $q \in X$. Then
$(I-U)[h-(I-V)q]=0$, and since $U$ has no fixed points, $h=(I-V)q$, which yields
$q=\sum_{n=0}^\infty V^n h$. Hence 
$$q_{j,k} =\sum_{n=0}^\infty h_{j,k+n} = \sum_{n=0}^\infty \frac1{(j+k+n)^{a/p}} =
\sum_{n=j+k}^\infty \frac1{n^{a/p}} \ge  \frac{c}{(j+k)^{(a-p)/p}}.
$$
Since $a-p=1$, we have $\sum_{k=1}^\infty |q_{j,k}|^p = \infty$, so $q \notin X$, a contradiction; 
hence $f$ is not a double coboundary. Note that the solution $q$ is in the larger space
$\ell_r(\mathbb N^2)$ for $r >2p$.
\smallskip

Now let $p=1$, and define $h$ by $h_{j,k} = 1/\big[(j+k)^2[\log(j+k)]^2\big]$. Then $h \in X$, since
$$
\sum_{j=1}^\infty\sum_{k=1}^\infty h_{j,k} \le 
\sum_{j=1}^\infty \sum_{k=1}^\infty \frac1{(j+k)^2[\log(j+1)]^2} =
$$
$$
\sum_{j=1}^\infty \frac1{[\log(j+1)]^2}\sum_{k=j+1}^\infty \frac1{k^2} \le 
\sum_{j=1}^\infty \frac1{j[\log(j+1)]^2} < \infty.
$$
As before, $Vh=Uh$ and $f=(I-U)h$ is a joint coboundary. If $f =(I-U)(I-V)q$, then, as before,
$q = \sum_{n=0}^\infty V^n h$, so
$$
q_{j,k} = \sum_{n=0}^\infty h_{j,k+n} = \sum_{n=j+k}^\infty \frac1{n^2[\log n]^2} \ge
\sum_{n=j+k}^\infty \frac1{n^{5/2}} \ge \frac{c}{(j+k)^{3/2}}
$$
when $j>J_0$. Then $\sum_{k=1}^\infty q_{j,k} \ge c \sum_{k=j+1}^\infty k^{-3/2} \ge c'j^{-1/2}$
for $j>J_0$, so $q \notin X$. This contradiction shows that $f$ is not a double coboundary.
Since $q_{j,k} \le \sum_{n=j+k}^\infty n^{-2} \le (j+k-1)^{-1}$, we obtain that 
$q \in \ell_r(\mathbb N^2)$ for $r >2$.
\smallskip

\bigskip

\section{Joint coboundaries of irrational rotations of the circle }

In this section we look at two irrational rotations of the unit circle $\mathbb T$. In 
this case we refine Theorem \ref{mpt}, showing first the existence of a joint coboundary 
in $C(\mathbb T)$ which is not a double coboundary, not even in $L_1(\mathbb T)$,
and then exhibit such a joint coboundary which is not even a measurable double coboundary.
\smallskip

Let $\alpha$ be a real number.  We denote $e(x):= \e^{2\pi i x}$. 
For $f \in L_1(\mathbb T)$ we define $T_\alpha f(z) = f(e(\alpha)z)$,
which preserves Lebesgue's measure on $\mathbb T$.  The operator $T_\alpha$ is a
contraction of all the $L_p(\mathbb T)$ spaces, $1 \le p \le \infty$, 
and is mean ergodic for $1\le p < \infty$. It is also a mean ergodic contraction of 
$C(\mathbb T)$. If $\beta$ is rational, then for some $k$
we have $T_\beta^k=I$, so $T_\beta$ is uniformly ergodic, and by Lemma \ref{same}
every joint coboundary of $T_\beta$ and any $T_\alpha$ in $C(\mathbb T)$ or in $L_p$, 
$1\le p<\infty$, is a double coboundary.

Let $\alpha$ and $\beta$ be irrational numbers.
Clearly, if $\alpha\equiv \beta \mod 1$ (i.e. $\alpha-\beta \in \mathbb Z$), 
then $T_\alpha=T_\beta$. Each $T_\alpha$ is invertible, with
$T_\alpha^{-1}=T_{-\alpha}$, and for any integer $n$, $T_{n\alpha} =T_\alpha^n$.

Since $I-T_\beta = T_\beta(T_\beta^{-1} -I)$, the double or joint coboundaries of
$T_\alpha, T_\beta$ are the same as the respective ones of $T_\alpha,T_\beta^{-1}$.

Let $\phi$ be a  centered trigonometric  polynomial $\sum_{|k|=1}^n a_k e(kx)$. 
Then for any irrational $\alpha$  we have 
$(I-T_\alpha) \sum_{|k|=1}^n \frac{a_k}{1-e(k\alpha)}e(kx) = \phi$. Hence $\phi$ is a 
double coboundary in $C(\mathbb T)$ for any two irrationals $\alpha $ and $\beta$. 

The  question for "how many" irrational rotations 
an $L_1(\mathbb T)$ function is a coboundary in $L_1(\mathbb T)$ was studied by 
Baggett et al. \cite{BMM}.

\begin{theo} \label{dependent}
Let $\alpha$ and $\beta$ be positive irrational numbers such that $\{\alpha,\beta,1\}$ are
linearly dependent over $\mathbb Z$ (i.e., there exist $m,n,p \in \mathbb Z$ with 
$m \ne 0$ and $m\alpha+n\beta+p=0$). If the g.c.d. of $(|m|,|n|)$ is 1,
then, in $C(\mathbb T)$ and in any $L_p(\mathbb T)$, the rotations $T_\alpha$ and $T_\beta$ 
have a joint coboundary which is not a double coboundary.
\end{theo}
\begin{proof} 
Since $\alpha$ and $\beta$ are irrational, also $n \ne 0$, and we may assume $n>0$. 
Since the g.c.d. of $|m|$ and $|n|$ is 1, there are integers $j,k$ such that $1=jm+kn$. Hence
$$
0= m\alpha+n\beta + (jm+kn)p = m(\alpha+ jp) + n(\beta+ kp).
$$
Put $\gamma:= \frac{\alpha+jp}n=-\frac{\beta+kp}m$. Then $T_\alpha =T_\gamma^n$ and 
$T_{\beta}^{-1}=T_{-\beta }= T_\gamma^m$. If $m>0$, by Theorem \ref{powers}
$T_\alpha$ and $T_\beta^{-1}$  (and therefore $T_\alpha$ and $T_\beta$) have 
a joint coboundary which is not a double coboundary; if $m <0$, then $T_\beta=T_\gamma^{|m|}$,
 and then $T_\alpha $ and $T_\beta$ have a joint coboundary which is not a double coboundary.
\end{proof}

{\bf Remarks.} 1. The assumptions of the theorem are equivalent to assuming that for some 
$\alpha'\equiv \alpha$ and $\beta'\equiv \beta$ we have $\alpha'/\beta'$ rational.
Indeed, if $(\alpha + j)/(\beta+k)$ is a rational $n/m$, then $m\alpha -n\beta + (jm-kn)=0$,
and g.c.d. of $|m|$ and $|n|$ divides $p:=jm-kn$. The converse implication is shown in 
the proof of the theorem.

2. Theorem \ref{dependent} is a special case of Theorem \ref{irrational-pair}, but its proof
follows from general principles.
\smallskip

Theorem \ref{dependent} shows that for every irrational $\alpha \in (0.1)$ there are 
countably infinitely many $\beta$, necessarily irrational (of the form $\beta=r(\alpha+j) +k$,
$\ r$ rational and $j,k$ integers), such that $T_\alpha$ and $T_\beta$ have a joint coboundary 
in $L_p$ ($1 \le p<\infty$)  which is not a double coboundary.

\begin{lem} \label{constants}
Let $T$ and $S$ be induced by commuting ergodic measure preserving transformations
of a probability space $(\Omega,\P)$. If $h$ is  measurable such that
$(I-T)(I-S)h=0$, then $h$ is constant a.e.
\end{lem}
\begin{proof}
Since $T$ is ergodic, $(I-S)h$ is a constant, 
hence integrable. By Anosov \cite[Theorem 1]{An} 
$\int_\Omega (I-S)h=0$, so $(I-S)h=0$;
by ergodicity of $S$, $h$ is constant a.e.
\end{proof}

For rotations, Lemma \ref{constants} is more general than (and independent of) Proposition 
\ref{homogeneous}(i).

\begin{prop} \label{Anosov}
Let $\alpha$ and $\beta$ be irrational numbers. Then there exist a continuous 
$\phi \in (I-T_\alpha)C(\mathbb T)$ and a measurable non-integrable $h$ such that 
$(I-T_\alpha)(I-T_\beta)h= \phi$, and every measurable $u$ satisfying
$(I-T_\alpha)(I-T_\beta)u= \phi$ is non-integrable.
\end{prop}
\begin{proof} By Anosov \cite[Theorem 2]{An} (see also \cite{Ko1}, \cite[Theorem 3]{BMM}), 
there exist $f \in C(\mathbb T)$  and $h$ measurable non-integrable with $(I-T_\beta)h=f$. We put 
$$
\phi =(I-T_\alpha)f =(I-T_\alpha)(I-T_\beta)h.
$$
Let $u$ be measurable and satisfy $(I-T_\alpha)(I-T_\beta)u=\phi$. Then $u-h=const$
by Lemma \ref{constants}, so $u=h+const$ is not integrable, since $h$ is not.
\end{proof}

The following corollary strengthens a particular case of Theorem \ref{dependent}.

\begin{cor} \label{anosov}
Let $\alpha$ and $\beta$ be irrational numbers, such that $T_\alpha=T_\beta^k$ for some 
$k\in\mathbb Z$. Then there exist continuous functions $f$ and $g$ such that 
$(I-T_\alpha)f=(I-T_\beta)g$, but there is no integrable $h$ with
 $(I-T_\alpha)(I-T_\beta)h=(I-T_\alpha)f$.
\end{cor}
\begin{proof}
We prove for $k >0$. Let $h$ and $f$ be as in the proof of Proposition \ref{Anosov}. Then
$\phi=(I-T_\alpha)f =(I-T_\beta)g$ with $g =\sum_{j=0}^{k-1}T_\beta^jf \in C(\mathbb T)$.
Proposition \ref{Anosov} shows that any $u$ with $(I-T_\alpha)(I-T_\beta)u=(I-T_\alpha)f$ is
non-integrable.
\end{proof}

{\bf Remark.} Proposition \ref{Anosov} and Corollary \ref{anosov} hold also when
we replace $T_\alpha$ and $T_\beta$ by $T$ and $S$ induced by commuting uniquely ergodic 
homeomorphisms of a compact metric space; the proofs are similar, with \cite[Theorem 2]{An}
replaced by its extension  by Kornfeld \cite{Ko1}.
\medskip
\smallskip

Our main purpose now is to study the existence of a joint coboundary for $T_\alpha$ and 
$T_\beta$ which is not a double coboundary, when $\alpha$ and $\beta$ are irrationals (which
are not rationally dependent). Wintner \cite{Wi} seems to have been the first to study coboundaries 
of rotation operators $T_\alpha$ for $\alpha$ irrational (using Fourier series methods).

As usual,  we denote by $\{\cdot\}$ the fractional part.  For a real number $\alpha$ we 
denote by $\|\alpha\|$ its distance from the nearest integer, so 
$\|\alpha\|=\min(\{\alpha\},1-\{\alpha\})$.
\smallskip

\begin{theo} \label{irrational-pair}
Let $\alpha$ and $\beta$ be irrational numbers.  Then: 

(i) There exists a joint coboundary for $T_\alpha$ and $T_\beta$ acting in $C(\mathbb T)$
which is not a double coboundary, not even in $L_1(\mathbb T)$.

(ii) For fixed $p \in[1,\infty)$, there exists a joint coboundary for $T_\alpha$ and 
$T_\beta$ acting in $L_p(\mathbb T)$ which is not a double coboundary in $L_1(\mathbb T)$.
\end{theo}
\begin{proof}
Since $C(\mathbb T) \subset L_p(\mathbb T) \subset L_1(\mathbb T)$ for $1<p< \infty$,
both parts of the theorem will follow if we produce continuous functions $f$ and $g$ with
$(I-T_\alpha)f =(I-T_\beta)g$ (a joint coboundary in $C(\mathbb T)$), such that
there is no $h \in L_1(\mathbb T)$ satisfying $(I-T_\alpha)f =(I-T_\alpha)(I-T_\beta)h$.
Remember that if $\sum_{n \in \mathbb Z} |a_n| < \infty$, then $f(z):=\sum_n a_n z^n$ is
continuous on $\mathbb T$.
\smallskip

By the two-dimensional Dirichlet theorem \cite[Theorem 200]{HW}, there are infinitely many 
positive integers $q$ such that $\max\{\|q\alpha\|,\|q\beta\|\} < \frac1{\sqrt{q}}$.
We take an increasing subsequence $(q_k)$ of these $q$, such that $\sum_k \frac1{\sqrt{q_k}}$
converges.

Joint coboundaries of $T_\alpha$ and $T_\beta$ in $C(\mathbb T)$ are given by continuous
 functions $f,g $ for which $(I-T_\alpha)f=(I-T_\beta)g$. 
This equality implies the following relations between the Fourier coefficients of $f$ and $g$: 

\begin{equation} \label{relation}
(1-{\rm e}^{2\pi i n\alpha})\hat f_n= (1-{\rm e}^{2\pi i n\beta})\hat g_n \qquad 
\text{for every } n \in \mathbb Z.
\end{equation}
We define $\hat f_n  =\hat g_n=0$ for $n\not\in (q_k)$, and put $\hat f_{q_k} =\|q_k\beta\|$.
By the choice of $(q_k)$, we have $f \in C(\mathbb T)$.
We then define $\hat g_{q_k}$ by the relation (\ref{relation}).

Since $\frac{\sin (\pi x)}{\pi x}$ is positive and decreases on $(0,1/2)$ and tends to 1 
as $x \to 0^+$, we have $\frac2\pi \le \frac{\sin (\pi x)}{\pi x}\le 1$ for $0<x\le \frac12$. 
We then obtain, 
$$
\sum_{n}|\hat g_n|=
\sum_k |\hat f_{q_k}| \frac{|1-{\rm e}^{2\pi i q_k\alpha}|}{|1-{\rm e}^{2\pi i q_k\beta}|}=
\sum_k |\hat f_{q_k}| \frac{|\sin(q_k\pi\alpha)|}{|\sin(q_k\pi\beta)|} =
$$ 
$$
(*)\qquad \sum_k |\hat f_{q_k}| \frac{|\sin(\pi\|q_k\alpha\|)|}{|\sin(\pi\|q_k\beta\|)|}\le
\frac\pi 2 \sum_k |\hat f_{q_k}| \frac{\|q_k\alpha\|}{\|q_k\beta\|}\le
\frac\pi 2 \sum_k \|q_k\alpha\| \le \frac\pi 2 \sum_k \frac1{\sqrt{q_k}} <\infty.
$$
Thus we have found $g \in C(\mathbb T)$ with $(I-T_\beta)g=(I-T_\alpha)f$.

Now, we want to show that  $(I-T_\alpha)f$ is not a double coboundary even in $L_1$: 
that is, there is no $h\in L_1$ such that $(I-T_\alpha)f=(I-T_\alpha)(I-T_\beta)h$ holds. 
Suppose there is; it then implies the following restrictions on the corresponding Fourier 
coefficients of $h$: 
$$
(1-{\rm e}^{2\pi i n\alpha})\hat f_n=
(1-{\rm e}^{2\pi i n\alpha})(1-{\rm e}^{2\pi i n\beta})\hat h_n  \qquad  \text{for every }n.
$$ 
The same computation as above, with $\hat f_n=0$ for $n \notin (q_k)$,  yields:
$$
(**)\qquad |\hat h_{q_k}| =   \frac{|\hat f_{q_k}|}{|1-\e^{2\pi iq_k\beta}|}=
\frac{ |\hat f_{q_k}|}{2|\sin(\pi\|q_k\beta\|)|}\ge 
 \frac{ |\hat f_{q_k}|}{2\pi \|q_k\beta\|}= \frac1{2\pi}.
$$
Since $h \in L_1$, the Riemann-Lebesgue lemma yields $|\hat h_{q_k}|\to 0$, which is
 a contradiction.
\end{proof}

\begin{cor} \label{no-measurable}
Let $\alpha$ and $\beta$ be irrational numbers. Then there exist continuous functions 
$f,g \in C(\mathbb T)$ such that  $(I-T_\alpha)f=(I-T_\beta)g$, but there is no measurable $h$
satisfying $(I-T_\alpha)(I-T_\beta)h=(I-T_\alpha)f$ 
(i.e. $(I-T_\alpha)f$ is not a measurable double coboundary).
\end{cor}
\begin{proof} In the construction of $f$ and $g$ in the proof of Theorem \ref{irrational-pair}, 
we can assume, by taking a subsequence, that $(q_k)$ is lacunary: for some $Q>1$, we have
 $ q_{k+1}/q_k \ge Q$ for every $k$. Then $f$ and $g$ are continuous with lacunary Fourier 
series, and  the same holds for $(I-T_\alpha)f$.  Without loss of generality, 
we may assume $\int_{\mathbb T} f =0$.

If $h$ is measurable with $(I-T_\alpha)(I-T_\beta)h=(I-T_\alpha)f$, then 
$ (I-T_\alpha)[f-(I-T_\beta)h]=0$, and ergodicity of $T_\alpha$ implies that 
$f=(I-T_\beta)h +const$. Hence $(I-T_\beta)h$ is continuous, and by Anosov \cite[Theorem 1]{An}
$\int_{\mathbb T} (I-T_\beta)h = 0$. Since $\int_{\mathbb T}f=0$, the constant is zero, 
and $f=(I-T_\beta)h$. But since $f$ is continuous with lacunary Fourier series, Herman's 
theorem \cite{He} says that $h \in L_2(\mathbb T)$,  contradicting Theorem \ref{irrational-pair}.
\end{proof}

J.P. Conze has noted that the work of Conze and Marco \cite{CM} yields an interesting result,
in the spirit of Corollary \ref{no-measurable}, with very simple $L_2$ joint coboundaries.

\begin{prop} \label{conze}
For any irrational $\alpha \in (0,1)$ with unbounded quotients there exist uncountably 
many pairs $(\beta,\gamma)$ of irrational numbers in $(0,1)$, such that $\{1,\gamma,\alpha\}$ 
are linearly independent over $\mathbb Q$, and $(I-T_\beta)1_{[0,\gamma]}$ is an $L_2$ joint 
coboundary of $T_\alpha$ and $T_\beta$ which is not a measurable double coboundary.
\end{prop}
\begin{proof}
By \cite[Theorem 2.2]{CM}, if $\alpha$ has unbounded partial quotients, then there is an 
uncountable set of pairs of irrational numbers $\beta$ and $\gamma$ in $(0, 1)$ such that 
$\varphi_{\beta,\gamma}:= 1_{[0, \gamma]} - T_\beta 1_{[0, \gamma]}\in (I-T_\alpha)L_2$,
i.e. $\varphi_{\beta,\gamma} = (I- T_\alpha) \psi$ with $\psi=\psi_{\beta,\gamma}$ in $L_2$. 
Hence $f = 1_{[0, \gamma]} - \gamma$  satisfies $(I- T_\beta) f =(I-T_\alpha)\psi$. 
Since the number of pairs $(\beta,\gamma)$ is uncountable, there is an uncountable subset of 
pairs with $\{1,\alpha,\gamma\}$ linearly independent over the rationals. For such pairs,
by Oren \cite{Or}, the skew product $T_f(x,y)=(x+\alpha,y + f(x))$ on $\mathbb T \times \mathbb R$ 
is ergodic.  We prove that there is no measurable $h$ such that 
$(I-T_\alpha) (I- T_\beta) h = (I-T_\beta)f$;
indeed, if such $h$ existed, we would have $f = (I- T_\alpha) h$  (shown similarly to the proof 
of Corollary \ref{no-measurable}).  Then the set $\bar E(f)$ of essential values of $f$ is 
$\{0\}$ \cite[Theorem 3.9(4)]{Sc}. But ergodicity of $T_f$ yields that 
$\mathbb R \subset \bar E(f)$ \cite[Corollary 5.4]{Sc} -- a contradiction.  
\end{proof}

{\bf Remarks.} 1. In the proof it is shown that {\it if $\alpha,\gamma \in(0,1)$ with 
$\{1,\alpha,\gamma\}$ linearly independent over $\mathbb Q$, then 
$f= 1_{[0,\gamma]} -\gamma$ is not a measurable coboundary of $T_\alpha$.} 
Petersen's result \cite{Pet} shows only $f \notin (I-T_\alpha)L_2$.
\smallskip

2. An irrational $\alpha $ has unbounded partial quotients if and only if
$\liminf_n n\|n\alpha\|=0$ (e.g. \cite[Section 11.10]{HW}). The set of such $\alpha$ in 
$[0,1]$ has  Lebesgue measure 1, as its complement in $[0,1]$ 
has measure zero \cite[Theorem 196]{HW}.
Thus Proposition \ref{conze} applies to almost every (irrational) $\alpha \in (0,1)$.
\medskip

The following result concerning measurable joint coboundaries is a special case of a result of
Conze and Marco \cite[Proposition 1.5]{CM}. An example is given in \cite[Theorem 2.1]{CM}.

\begin{prop} \label{cm}
Let $u$ be a measurable function on $\mathbb T$ and $\alpha \in (0,1)$ irrational. If the set of 
$\beta \in (0,1)$, for which $(I-T_\beta)u =(I-T_\alpha)v_\beta$ for some measurable $v_\beta$,
has positive measure, then there exist a measurable $h$ and a constant $C$ such that 
$u=(I-T_\alpha)h +C$. If $u$ is integrable, then $C =\int_{\mathbb T}u$.
\end{prop}

Let $u$ and $\alpha$ be as in Proposition \ref{cm}. Then $u-C$ satisfies the same assumptions, 
so we may assume $C=0$ (even without integrability of $u$). For each irrational $\beta$ as in 
the proposition we then have that the joint coboundary $(I-T_\beta)u= (I-T_\alpha)v_\beta$ is 
a (measurable) double coboundary: 
$(I-T_\beta)u=(I-T_\beta)(I-T_\alpha)h$; if $h$ is non-integrable, then any $w$ satisfying 
$(I-T_\alpha)(I-T_\beta)w= (I-T_\beta)u$ is non-integrable, by Lemma \ref{constants}.

\begin{prop} \label{ks}
Let $\phi \in C(\mathbb T)$ with $\int_{\mathbb T} \phi=0$ have Fourier coefficients satisfying 
$\sum_k |\hat \phi_k| < \infty$ and $\sum_{|k|>0} |\hat\phi_k| \log(1/|\hat\phi_k|) < \infty$. 
Then $\phi \in (I-T_\alpha)C(\mathbb T)$  for almost every $\alpha$. 
Hence for almost every pair $(\alpha,\beta)$, there exist $f,g \in C(\mathbb T)$ such that 
$(I-T_\alpha)f= \phi =(I-T_\beta)g$ ($\phi$ is a joint coboundary in $C(\mathbb T)$).
\end{prop}
\begin{proof}
By Kac and Salem \cite{KS}, the series 
\begin{equation} \label{ks-series}
\sum_{|k|=1}^\infty \frac {|\hat\phi_k|}{|\sin(\pi kx)|} 
\end{equation}
converges a.e.  For $x=\alpha$ for which the series converges,
define $c_0=0$ and $c_k:= \hat\phi_k /2\sin(\pi k\alpha)$ for $|k| >0$. 
Then  $\sum_k |c_k| <\infty$, and the function $f(z) :=\sum_k c_k z^k$ is in 
$C(\mathbb T)$, and satisfies $(I-T_\alpha)f=\phi$.
\end{proof}

{\bf Remarks.} 1. Proposition \ref{ks} applies when 
$|\hat\phi_k| = O(1/|k|(\log|k|)^{2+\epsilon})$. 
This improves a result of Herman \cite[p. 230, Proposition 8.2.1]{He0}.

2. It can be shown that the conditions $\hat \phi_0=0$ and
 $\sum_{|k|>0}|\hat\phi_k| \log|k| < \infty$ imply that the conditions of Proposition 
\ref{ks} hold, so $\phi \in (I-T_\alpha)C(\mathbb T)$ for almost every $\alpha$.
Muromski\v\i \ \cite{Mu} proved  a.e. convergence of \eqref{ks-series} under 
these conditions.
 
\medskip

{\bf Definition.} A (necessarily irrational) real number $\alpha$ is said to be {\it badly 
approximable} ({\it bad}  for short) if there exists $c>0$ such that
$$
\|q\alpha\| := \min\{|q\alpha -p|: p\in \mathbb Z\} > \frac{c}q \qquad \forall q \in \mathbb N.
$$
The set of badly approximable numbers is known to have Lebesgue measure 0 and Hausdorff dimension 1.
Rozhdestvenski\v\i\  \cite{Ro} constructed mean zero $L_2$ functions such that whenver
$\alpha$ is bad, there is no measurable $h$ satisfying $(I-T_\alpha)h= f$.

\begin{lem} \label{bad-joint}
Let the Fourier coefficients of $f$ satisfy $\hat f_0=0$.

(i) If  $\sum_{k \ne 0}|k|\cdot |\hat f_k|<\infty$,
then $f \in (I-T_\alpha)C(\mathbb T)$ for every badly approximable $\alpha$.

(ii) If  $\sum_{k \ne 0}|k|^2 |\hat f_k|^2 <\infty$,
then $f \in C(\mathbb T) \cap (I-T_\alpha)L_2(\mathbb T)$ for every badly approximable $\alpha$.
\end{lem}
\begin{proof} We prove (i):
$$
2 \sum_{|k|=1}^\infty \frac {|\hat f_k|}{|1-\e^{2\pi ik\alpha}|} =
\sum_{|k|=1}^\infty \frac {|\hat f_k|}{|\sin(\pi k\alpha)|} \le
\sum_{|k|=1}^\infty \frac {|\hat f_k|}{2\| k\alpha\|} \le
\sum_{|k|=1}^\infty \frac {|k|}{2c}|\hat f_k| < \infty.
$$
Hence the function $g(z)=\sum_{|k|=1}^\infty \frac{\hat f_k}{1-\e^{2\pi ik\alpha}} z^k$ is in
$C(\mathbb T)$ and satisfies $(I-T_\alpha)g=f$.

(ii) By Cauchy-Schwarz, $\sum_k |\hat f_k|<\infty$. A computation similar to (i)
yields that the above $g$ is in $L_2(\mathbb T)$.
\end{proof}


\begin{prop} \label{mur}
Let $a_k \downarrow 0$ 
satisfy $\sum_{k=1}^\infty k a_k^2< \infty$. 
If the Fourier coefficients of $f \in L_2(\mathbb T)$ satisfy $\hat f_0=0$ and 
$|\hat f_{|k|}|=O(a_{|k|})$ for $k\ne 0$, then $f \in (I-T_\alpha)L_2 $ for every badly approximable  
$\alpha$.
\end{prop}
\begin{proof}
When $a_k \downarrow 0$, then by a result of Muromski\v\i \ \cite{Mu}, (with $\alpha=2$ and 
$c_k=a_k^2$), we have 
$$
\sum_{|k|=1}^\infty \frac {|\hat f_k|^2}{|\sin(\pi kx)|^2} \le
C \sum_{|k|=1}^\infty \frac {|a_k|^2}{|\sin(\pi kx)|^2}  <\infty,
$$
for every $x$ having a continued fraction expansion with bounded elements (quotients). 
It is known (e.g. \cite[Section 11.10]{HW}) that badly approximable numbers have bounded
quotients.
\end{proof}

{\bf Remarks.} 1. Proposition \ref{mur} applies to $f \in L_2$ with 
$|\hat f_k| = O(1/|k|(\log|k|)^{\frac12+\delta})$.

2. Let  $\hat f_k = 1/k(\log k )^{\frac12+\delta}$ for $k>0$ and $\hat f_k=0$ for $k \le 0$.
Then $f \in (I-T_\alpha)L_2$ for any bad $\alpha$, $\sum_k k|f_k|^2< \infty$, but 
$\sum_k k^2|f_k|^2 = \infty$.

\begin{cor} \label{double-mur}
Let $f \in L_2(\mathbb T)$ satisfy $\hat f_0=0$ and $|\hat f_{k}|=O(1/k^2(\log |k|)^\gamma)$
with $\gamma>1$. Then for any badly approximable numbers $\alpha $ and $\beta$, 
$f \in (I-T_\alpha)C(\mathbb T)\cap (I-T_\beta)C(\mathbb T)$ and 
$f \in (I-T_\alpha)(I-T_\beta)L_2(\mathbb T)$.
\end{cor} 
\begin{proof}
By Lemma \ref{bad-joint}, $f \in (I-T_\alpha)C(\mathbb T)\cap (I-T_\beta)C(\mathbb T)$. Let 
$g(z)=\sum_{|k|=1}^\infty \frac{\hat f_k}{1-\e^{2\pi ik\alpha}} z^k$, which satisfies $(I-T_\alpha)g=f$.
 For $k>0$ put $a_k=1/k(\log k)^\gamma$, so $|\hat g_{|k|}| =O(a_{|k|})$. Then $\{a_k\}$ satisfies 
the assumptions of Proposition \ref{mur}, so $g \in (I-T_\beta)L_2(\mathbb T)$, which shows that
$f \in (I-T_\alpha)(I-T_\beta)L_2(\mathbb T)$.
\end{proof}

{\bf Remark.} When the Fourier coefficients of $f$ satisfy $\hat f_0=0$ and  the stronger condition
$\sum_{k \ne 0}k^2 |\hat f_k|<\infty$, we define, for $\alpha$ and $\beta$ bad,
 $c_k= \hat f_k/(1-\e^{2\pi ik\alpha})(1-\e^{2\pi ik\beta})$ for $k \ne 0$ and $c_0=0$. 
Similarly to the proof of Lemma \ref{bad-joint}, we obtain that $\sum_k |c_k|<\infty$, and then 
$h(z)=\sum_k c_kz^k$ satisfies $f= (I-T_\alpha)(I-T_\beta)h$.

\begin{prop} \label{large-coeff}
Let the Fourier coefficients of $f \in L_2(\mathbb T)$ satisfy 
$\liminf_{|n|\to\infty} |n\hat f_n|>0$. Then:

(i) For any $\beta$ irrational, $f \not\in (I-T_\beta)L_1$
(so $f$ is not a joint coboundary of $T_\alpha$ and $T_\beta$).

(ii) For $\beta$ irrational, if $(I-T_\beta )f$ is a joint coboundary in $L_2$ with $T_\alpha$,
i.e. $(I-T_\beta)f \in (I-T_\alpha)L_2$, then  $\beta =k\alpha +n$ with $k,n \in \mathbb Z$
(i.e. $T_\alpha^k =T_\beta$).

(iii) If $\alpha$ is irrational and $T_\beta=T_\alpha^k$, then $(I-T_\beta)f$ is a joint coboundary
of $T_\alpha$ and $T_\beta$ which is not a double coboundary in $L_1$.
\end{prop}
\begin{proof} (i) Assume $f =(I-T_\beta)h$ with $h \in L_1$. Then, as in (**), 
$|\hat h_n| \ge  \frac{|\hat f_n|}{2\pi\|n\beta\|}$. For $n>N$ and some $C>0$ we have
$$
\frac{C}{2\pi |n| \|n\beta\|}< C \frac{|\hat f_n|}{2\pi\|n\beta\|}= {C} |\hat h_n|\to 0,
$$
using the Riemann-Lebesgue lemma. Hence $\liminf_{n\to\infty} n\|n\beta\| =\infty$, 
a contradiction to Dirichlet's theorem \cite[Theorem 185]{HW}, which yields
$\liminf_{n\to\infty} n\|n\beta\| \le 1$.
\smallskip

(ii) Let $g \in L_2$ satisfy $(I-T_\alpha)g=(I-T_\beta)f$. Computing Fourier coefficients we obtain
$$
\sum_{n \ne 0} |\hat f_n|^2 \frac {|\sin(n\pi \beta)|^2}{|\sin(n\pi \alpha)|^2} =
\sum_n |\hat g_n|^2  < \infty,
$$
so the assumption $\liminf_{|n|\to\infty} |n\hat f_n|>0$ yields
$\sum_{n \ne 0} \frac1{n^2} \frac {|\sin(n\pi \beta)|^2}{|\sin(n\pi \alpha)|^2} < \infty$.
By Petersen \cite{Pet}, $\beta \in \mathbb Z\alpha \, \mod 1$.
\smallskip

(iii)  We may assume $k>0$. Since $I-T_\alpha^k =(I-T_\alpha) \sum_{j=0}^{k-1} T_\alpha^j$,
$\ (I-T_\beta)f$ is a joint coboundary, and we may assume $\int_\mathbb Tf=0$. If there is 
$h \in L_1$ with $(I-T_\alpha)(I-T_\beta)h = (I-T_\beta)f$, then $f-(I-T_\alpha)h$ is a constant, 
which is $\int_\mathbb T f=0$. This contradicts part (i); hence $(I-T_\beta)f$ is not a double
coboundary in $L_1$.
\end{proof}

{\bf Remarks.} 1. If $f \in L_2$ satisfies $\liminf_{|n| \to\infty} |n\hat f_n| >0$, then by 
\cite{DKK} (see also \cite[p. 283]{Ka}) there exists a function $\phi \in C(\mathbb T)$ with 
$\liminf_{|n| \to\infty} |n\hat \phi_n| \ge \liminf_{|n| \to\infty} |n\hat f_n| >0$. 
\smallskip

2. Compared with Theorem \ref{joint1}, Proposition \ref{large-coeff}(iii) yields an explicit 
construction of joint coboundaries in $L_2$ which are not double coboundaries (even in $L_1$).
It also yields, via the above mentioned result of \cite{DKK}, joint coboundaries in $C(\mathbb T)$
which are not double coboundaries in $L_1$  (see Corollary \ref{anosov}). 
\smallskip

3. Part (ii) of Proposition \ref{large-coeff} proves the following special case of
Kornfeld's result \cite{Ko}: {\it If $\alpha$ and $\beta$ are irrational and 
$(I-T_\beta)L_2(\mathbb T) \subset (I-T_\alpha)L_2(\mathbb T)$, then $T_\beta =T_\alpha^k$ for 
some $ k\in\mathbb Z$.}

\begin{prop} \label{squares}
Let the Fourier coefficients of $f \in L_2(\mathbb T)$ satisfy 
\begin{equation}
 C:=\liminf_{n\to\infty} n^\delta |\hat f_{n^2}| >0 \qquad 
\text{for some fixed}\quad \delta \in (\frac12,\frac23).
\end{equation}
Then for any $\beta$ irrational, $ f \not\in (I-T_\beta)L_1$.
\end{prop}
\begin{proof} Assume $f =(I-T_\beta)h$ with $h \in L_1$. Then, as in (**), 
$|\hat h_n| \ge  \frac{|\hat f_n|}{2\pi\|n\beta\|}$. By Zaharescu \cite[Theorem 1]{Za},
there exists an increasing subseequence $(n_k)$ with $\|n_k^2\beta\| < n_k^{-\delta}$. 
Then for $n_k>N$ we have
$$
|\hat h_{n_k^2}| \ge \frac{|\hat f_{n_k^2}|}{2\pi\|n_k^2\beta\|} \ge 
\frac{|\hat f_{n_k^2}|n_k^\delta}{2\pi} \ge \frac{C}{3\pi},
$$
which contradicts the Riemann-Lebesgue lemma. Hence $f \not\in (I-T_\beta)L_1$.
\end{proof}

{\bf Remarks.} 1. The requirement $\delta >0.5$ follows from $f \in L_2$.

2. Propositions \ref{large-coeff} and \ref{squares} are not comparable. 
In Proposition \ref{squares} we may have  $\hat f_n =0$ for infinitely many  $n>0$, while 
in Proposition \ref{large-coeff}(i), which holds also if $\liminf_{n\to\infty} n|\hat f_n|>0$,
this assumption implies $\hat f_n \ne 0$ from some place on. The price we pay in Proposition 
\ref{squares} is that the coefficients at $n^2$ have to be larger, of order $1/n^\delta$
(instead of $1/n^2$).
\medskip

{\bf Definition.} A pair $(\alpha,\beta)$ of irrational numbers is said to be 
{\it badly approximable} if
\begin{equation} \label{bad}
C(\alpha,\beta):= \liminf \sqrt{q} \max\{\|q\alpha\|,\|q\beta\|\} >0.
\end{equation}
The set ${\bf Bad}_2$ of bad pairs is not empty, by Perron \cite{Pe}. A consequence of 
Khintchine's theorem is  that it has Lebesgue measure  zero. ${\bf Bad}_2$ is uncountable, 
since it has maximal Hausdorff dimension 2 \cite{PV}. Moreover, for a given bad $\alpha$, 
$\ \dim_H\{\beta: (\alpha,\beta) \in {\bf Bad}_2\} =1$, see \cite[p. 1840]{BPV}.

For badly approximable pairs of irrationals, the method of proof of Theorem 
\ref{irrational-pair} and Corollary \ref{no-measurable} allows the construction of many 
more joint coboundaries in $C(\mathbb T)$ which are not  measurable double coboundaries.

\begin{prop}
Let $(\alpha,\beta)$ be a badly approximable pair. Then there exists an infinite sequence 
of positive integers $(q_k)_{k \ge 1}$ with $\sum_k \frac1{\sqrt{q_k}} < \infty$, such that 
whenever  $\sum_k|a_k| < \infty$ and $\limsup\sqrt{q_k}|a_k|>0$, the function 
$f \in C(\mathbb T)$, with $\hat f_{q_k}=a_k$ and $\hat f_n=0$ for $n \not\in (q_k)$, yields 
a joint coboundary $(I-T_\alpha)f \in C(\mathbb T)$ which is not a measurable double coboundary.
\end{prop}
\begin{proof}
By the two-dimensional Dirichlet theorem, $C(\alpha,\beta) \le 1$. Let 
$$\lim_k \sqrt{q_k} \max\{\|q_k\alpha\|,\|q_k\beta\|\} =C(\alpha,\beta)=C>0.
$$
Without loss of generality, we may assume $\|q_k\beta\| \ge \|q_k\alpha\|$
for infinitely many $q_k$.  We take an increasing subsequence of these $q_k$, still
denoted by $(q_k)$, such that 
$$
\frac{C}2 \le \sqrt{q_k} \|q_k\beta\| \le 2C \quad \forall k>0, \quad \text{and }
\sum_k \frac1{\sqrt{q_k}} < \infty.
$$ 
We take a further subsequence, still denoted by $(q_k)$, such that 
$inf_k \frac{ q_{k+1}}{q_k}\ge Q >1$ ($(q_k)$ is lacunary). By the choice of $q_k$, we have
$\frac{\|q_k\alpha\|}{\|q_k\beta\|}\le 1$, so for any $f$ with $\hat f_{q_k}=a_k$ and 
$\hat f_n=0$ otherwise, we get, as in (*), that $g$ with $\hat g_n$ defined by (\ref{relation}) 
is in $C(\mathbb T)$, and $(I-T_\alpha)f=(I-T_\beta)g$. The condition $\limsup \sqrt{q_k}|a_k|$
is used to obtain that $(I-T_\alpha)f$ is not a double coboundary in $L_1$, by a contradiciton
to the Riemann-Lebesgue lemma, similar to (**). If $(I-T_\alpha)f=(I-T_\alpha)(I-T_\beta)h$,
then $(I-T_\beta)h=f$ (since $\hat f_0=0$).  Since $(q_k)$ is lacunary, by Herman \cite{He}
$h \in L_2$, contradicting the fact that $(I-T_\alpha)f$ is not an $L_1$ double cobundary.
\end{proof}
\bigskip

\section{Coboundary sums and uniform ergodicity of commuting contractions}

Let $\theta$ and $\tau$ be commuting ergodic measure preserving transformations of the probabilty 
space $(\Omega, \mathcal B,\mathbb P)$, and let $f \in L_2(\mathbb P)$ with 
$\int_\Omega f \ d\mathbb P=0$.  The central limit theorem (CLT) problem is to find conditions 
for the convergence in distribution of 
$\frac1n \sum_{k=0}^{n-1}\sum_{j=0}^{n-1} f(\theta^k\tau^j \omega)$. The Koopman operators
$Tg=g\circ \theta$ and $Sg=g\circ \tau$  commute, so the CLT problem is the convergence in 
distribution of $\frac1n \sum_{k=0}^{n-1}\sum_{j=0}^{n-1} T^kS^jf$. When this latter expression
converges in $L_2$-norm to zero, we have a degenerate CLT (a zero {\it asymptotic variance}).
This motivates the results of this section.

\begin{prop} \label{degenerate}
Let $T$ and $S$ be commuting mean ergodic contractions on a Banach space $X$ with $F(T)=F(S)$,
and let $z := (I-T)x+(I-S)y$. Then $\|\frac1n \sum_{k=0}^{n-1}\sum_{j=0}^{n-1} T^kS^jz\|$
converges to zero.
\end{prop}
\begin{proof}
It is enough to prove when $z=(I-T)x$.  Let $E_S:= \lim_n \frac1n\sum_{j=0}^{n-1} S^j$ (in the 
strong operator topology). Since $E_Sx$ is $S$-invariant it is also $T$-invariant, and  we may 
replace $x$ by $x -E_Sx$, so we assume $E_Sx=0$. 
Then
$$
\Big\|\frac1n \sum_{k=0}^{n-1}\sum_{j=0}^{n-1} T^kS^j z \Big\| =
\Big\| \frac1n \sum_{j=0}^{n-1} S^j (I-T^n)x \Big\| \le
\Big\|\frac1n \sum_{j=0}^{n-1} S^jx \Big\|+\Big\| T^n \big(\frac1n \sum_{j=0}^{n-1} S^jx \big) \Big\| 
\le 
$$
$$2\Big\| \frac1n \sum_{j=0}^{n-1} S^j x \Big\| \to \|E_Sx\|=0.
$$
\end{proof}

{\bf Remarks.} 1. If $F(S)\ne F(T)$, then for $x=Sx\ne Tx$, the proposition fails when
$T^nx \not\to x$. For example, on $X$ reflexive take $S=I$ and $T$ such $T^n \to 0$ in the
weak operator topology.

2. $z=(I-T)x+(I-S)y$ is a joint coboundary if and only if both $(I-T)x$ and $(I-S)y$ are.

3. If $T$ and $S$ are induced by commuting invertible ergodic probability preserving
transformations and $S \ne T^k$ for any $k \in \mathbb Z$, then \cite{Ko} there exists $y$ 
such that $(I-S)y \not\in (I-T)X$, so for any $x$, $\ z=(I-T)x+(I-S)y$ is not a coboundary 
of $T$, so $z$ is not a joint coboundary.

4. A special case of the result of Lind \cite{Li} is that if $T$ and $S$ are induced by 
commuting invertible probability preserving transformations such that $T^mS^n \ne I$ when 
$m \ne n$, then for every measurable $h$ there exist measurable $f$ and $g$ such that 
$h=(I-T)f +(I-S)g$. Proposition \ref{degenerate} requires $f$ and $g$ to be in the same
Banach space as $h$ (e.g $L_p$).

5. The proof of Proposition \ref{degenerate} can be easily modified to show that for any 
sector $\mathcal S:=\{(m,n) \in \mathbb N^2: 0<\alpha \le \frac mn \le \beta <\infty\}$ 
we have
$$ \lim_{n\wedge m \to \infty, \ (m,n) \in \mathcal S} 
\Big\| \frac1{\sqrt{mn}} \sum_{k=0}^{m-1}\sum_{j=0}^{n-1} T^kS^jz \Big\| =0.
$$
\smallskip

The following proposition is well-known, and its proof is similar to the proof for a single
operator; the Hahn-Banach theorem is used to show the "only if" in (i).

\begin{prop} \label{squares-se}
let $T$ and $S$ be commuting power-bounded operators on a Banach space $X$. Then:

(i) $\displaystyle{ \big\|\frac1{n^2}\sum_{k=0}^{n-1}\sum_{j=0}^{n-1} T^kS^jz \big\| \to 0}$ 
if and only if $z \in \overline{(I-T)X+(I-S)X}$.

(ii) $\displaystyle{ \frac1{n^2}\sum_{k=0}^{n-1}\sum_{j=0}^{n-1} T^kS^j}$ converges stronlgy
if and only if 
$$X=[F(T)\cap F(S)] \oplus \overline{(I-T)X+(I-S)X}.$$
\end{prop}

\begin{theo} \label{squares-ue}
Let $T$ and $S$ be commuting mean ergodic contractions on a Banach space $X$.
Then the following are equivalent:

(i) $(I-T)X+(I-S)X$ is closed in $X$.

(ii) $\displaystyle{ \frac1{n^2}\sum_{k=0}^{n-1}\sum_{j=0}^{n-1} T^kS^j}$ converges in 
operator-norm, as $n \to \infty$.

(iii) $\displaystyle{\frac1{nm} \sum_{k=0}^{n-1}\sum_{j=1}^{m-1}T^kS^j}$ converges in 
operator-norm  as $\min(n,m) \to \infty$.
\end{theo}
\begin{proof}
Assume (i), and put $Y:=\overline{(I-T)X+(I-S)X}$. By (i) $Y=(I-T)X+(I-S)X$.
 Fix $1< \alpha <2$. For $z=(I-T)x$ we have
$$
\Big\| \frac1{n^\alpha} \sum_{k=0}^{n-1}\sum_{j=0}^{n-1}T^kS^j z \Big\|=
\Big\|\frac1{n^\alpha} \sum_{j=0}^{n-1}S^j (I-T^n)x\Big\| \le
\Big\|\frac1{n} \sum_{j=0}^{n-1}S^j\Big\| \cdot\Big\|\frac{ (I-T^n)x}{n^{\alpha-1}}\Big\| \to 0.
$$
A similar computation for $z=(I-S)y$ shows that for $z \in (I-T)X+(I-S)X=Y$ we have
$ \frac1{n^\alpha} \sum_{k=0}^{n-1}\sum_{j=0}^{n-1}T^kS^j z \to 0$.  Hence
$\sup_n \big\|\frac1{n^\alpha} \sum_{k=0}^{n-1}\sum_{j=0}^{n-1}T^kS^j \big\|_Y < \infty$
by the Banach-Steinhaus theroem, which yields 
$\big\|\frac1{n^2} \sum_{k=0}^{n-1}\sum_{j=0}^{n-1}T^kS^j \big\|_Y \to 0$.
Since $T$ and $S$ are mean ergodic, by Proposition \ref{squares-se}(ii) (see also
\cite[Lemma 2.2]{CL}) we have $X=[F(T)\cap F(S)] \oplus Y$.
Let $E$ be the corresponding projection on $F(T)\cap F(S)$; then
$\big\|\frac1{n^2} \sum_{k=0}^{n-1}\sum_{j=0}^{n-1}T^kS^j -E \big\| \to 0$.

\smallskip

By Theorem \ref{mv-real}, (iii) and (ii) are equivalent, and each implies (i).
\end{proof}

{\bf Remark.} Unlike Theorem \ref{mv-real}, we do not need to assume in (i) that 
$(I-T^*)X^* + (I-S^*)X^*$ is closed in order to obtain (ii), because we have assumed
that $T$ and $S$ are mean ergodic.

\begin{theo} \label{jerome}
Let $T$ and $S$ be commuting mean ergodic contractions on a Banach space $X$ with $F(T)=F(S)$.
Then the following are equivalent:

(i) $(I-T)X+(I-S)X$ is closed.

(ii) For every $z \in \overline{(I-T)(I-S)X}$ we have
\begin{equation} \label{zero-var}
\Big\|\frac1n \sum_{k=0}^{n-1}\sum_{j=0}^{n-1} T^kS^jz\Big\| \to 0.
\end{equation}

(iii) $\displaystyle{\frac1{n^2} \sum_{k=1}^n\sum_{j=1}^nT^kS^j}$ converges in operator-norm. 

(iv) $\displaystyle{\frac1{nm} \sum_{k=0}^{n-1}\sum_{j=1}^{m-1}T^kS^j}$ converges in 
operator-norm  as $\min(n,m) \to \infty$.
\smallskip

When $(I-T)X+(I-S)X$ is closed, it equals $\overline{(I-T)(I-S)X}$, and the limit in (iii) is
the projection $E$ on $F(T)$ with $\ker(E) =\overline{(I-T)X}$.
\end{theo}
\begin{proof}
by \cite[Remark 2.5]{CL}, the assumption $F(T)=F(S)$ implies 
\begin{equation} \label{Closure}
\overline{(I-T)X}=\overline{(I-S)X}= \overline{(I-T)(I-S)X}.
\end{equation}
\smallskip

Assume (i). By (\ref{Closure}) 
$\overline{(I-T)X + (I-S)X} = \overline{(I-T)X}=\overline{(I-T)(I-S)X}$,
so when $(I-T)X+(I-S)X$ is closed Proposition \ref{degenerate}  yields  (\ref{zero-var})
for every $z \in \overline{(I-T)(I-S)X}$.
\smallskip

Assume (ii). Put $Y=\overline{(I-T)X}$, and assume that (\ref{zero-var}) holds for every 
$z \in Y$.  By (\ref{Closure}), $Y$ is $T$ and $S$ invariant, so we 
restrict ourselves to $Y$. Since
$\sup_n \|\frac1n \sum_{k=0}^{n-1}\sum_{j=0}^{n-1} T^kS^jz\| < \infty$ for every $z \in Y$,
 by the Banach-Steinhaus theorem
$\sup_n \|\frac1n \sum_{k=0}^{n-1}\sum_{j=0}^{n-1} T^kS^j\|_Y < \infty$. Hence
$\|\frac1{n^2} \sum_{k=0}^{n-1}\sum_{j=0}^{n-1} T^kS^j\|_Y  \to 0$. let $E$ be the ergodic 
projection of $T$ on $Y$; then (iii) holds with $E$ the limit.
\smallskip

By Theorem \ref{squares-ue}, (i), (iii) and (iv) are equivalent. 
\end{proof}


{\bf Remark.} If in the theorem $T$ is uniformly ergodic, then $(I-T)X+(I-S)X$ is closed, 
since $(I-T)X$ is closed, and
$$
(I-T)X \subset (I-T)X+(I-S)X \subset \overline{(I-T)X}=(I-T)X.
$$


\medskip

{\bf Examples.} 1. Let $\nu$ be an absolutely continuous probability on $\mathbb T$. By 
\cite[Theorem 3]{Bh} (see also \cite[Corollary 4.2]{AG}), $\|\nu^n -\lambda\| \to 0$ 
(in total variation norm), where $\lambda$ is the normalized Lebesgue measure on $\mathbb T$. 
Let $X=L_2(\mathbb T)$ and define $Tf=\nu*f$. Then $\|T^n-E\|_2 \to 0$ (where 
$Ef =\int f\,d\lambda$), so $T$ is uniformly ergodic. Let $S$ be induced on $L_2(\mathbb T)$ 
by a rotation by $\theta$ irrational. Then $F(T)=F(S)$ and $(I-T)X+(I-S)X$ is closed. 
Since $\sigma(S)=\mathbb T$, $\ S$ is not uniformly ergodic, so by Corollary 
\ref{equivalences} there is $z \in (I-T)X+(I-S)X =(I-T)X$ which is not a joint coboundary.
Note that by Lemma \ref{same}(i), every joint coboundary is a double coboundary.
\smallskip

2. On $[0,1)$ define $\theta t =2t \mod 1$ and $\tau t =3 t \mod 1$, and let $T$ and $S$
be the corresponding isometries induced on $L_2[0,1)$. Put $f(t)=\e^{2\pi i t}$. Then
orthogonality of the exponents yields
$$
\Big\| \frac1n \sum_{k=0}^{n-1}\sum_{j=0}^{n-1} T^kS^j f \Big\|_2^2 =
\frac1{n^2} \Big\| \sum_{k=0}^{n-1}\sum_{j=0}^{n-1}  \e^{2\pi i 2^k3^j t} \Big\|_2^2 = 1.
$$
Since $\int f(t)dt=0$, by Theorem \ref{jerome} $(I-T)L_2+(I-S)L_2$ is not closed in $L_2$.
\medskip

\begin{theo} \label{depauw}
Let $\theta$ and $\tau$ be commuting ergodic measure preserving transformations of a 
non-atomic probability space which generate a free $\mathbb N_0^2$ action.
For $1 \le p< \infty$, let $U$ and $V$ be their corresponding isometries induced on $L_p$. Then:

(i) $(I-U)L_p +(I-V)L_p$ is not closed (both for the real and for the complex $L_p$).

(ii) There exists a real function $f \in L_p$ with integral zero such that
$$\limsup_n \|\frac1n \sum_{k=0}^{n-1}\sum_{j=0}^{n-1} U^kV^j f\|_p >0.$$
\end{theo}
\begin{proof}
Assume that $(I-U)L_p+(I-V)L_p$ is closed, for the real $L_p$. Then it is also closed in 
the complex $L_p$, and we apply in that space Theorem  \ref{jerome}, which yields that
$\frac1{n^2}\sum_{k=0}^{n-1}\sum_{j=0}^{n-1} U^kV^j$ converges in operator norm, with 
the limit $E$ a projection onto $F:=F(U)=F(V)=\{complex \ constants\}$ with null space 
$Y:= (I-U)L_p +(I-V)L_p = \{f \in L_p: \int f=0\}$. 
By Corollary \ref{notin-sp}, $(1,1)$ is not in $\sigma(U_Y,V_Y)$, where $U_Y$ and 
$V_Y$ are the restrictions to $Y$.  The restrictions of $U$ and $V$ to $F(U)$ are both the 
identity, and for any complex Banach space $X$ we have $\sigma(I_X,I_X)=\{(1,1)\}$. 
Since $L_p=F(U)\oplus Y$,  by  (\ref{direct}) we have
$$
\sigma(U,V) \subset  \sigma(I_F,I_F) \cup \sigma(U_Y,V_Y) = \{(1,1)\} \cup \sigma(U_Y,V_Y).
$$
Since the joint spectrum is closed, there is a neighborhood of $(1,1)$ 
which is not in $\sigma(U_Y,V_Y)$, so $(1,1)$ is isolated in $\sigma(U,V)$, a contradiction 
to Corollary \ref{alex2}. Hence (i) holds in the complex $L_p$ and therefore also 
in the real $L_p$. Hence (ii) of Theorem \ref{jerome} fails, which yields (ii) of our theorem.
\end{proof}

{\bf Remarks.} 1. The research leading to Theorem \ref{depauw} was motivated by the result of
Depauw \cite[p. 168]{De}, who proved that for $U$ and $V$ induced on $L_2(\mathbb T)$ by two 
irrational rotations, there exists $f \in L_2$ with integral zero which cannot be represented 
as $f =(I-U)g+(I-V)h$ with $g,h \in L_2$; this is (i) of Theorem \ref{depauw} for $p=2$.
\smallskip

2. Theorem \ref{jerome} yields directly the result for $p=2$, without the theory of joint spectra,
when there exists $0\ne f \in L_2$ with integral zero, such that {\it all} the orbit is 
orthogonal, i.e. the functions $\{ U^kV^jf: k\ge 0,\ j\ge 0\}$ are orthogonal. 
See the second example following Theorem \ref{jerome}.
\medskip

Derriennic and Lin \cite{DL} introduced the notion of fractional coboundaries of contractions.
Let $0<a<1$, and let $(1-t)^a = 1 -\sum_{j=1}^\infty a_j t^j$. It is known that $a_j >0$ with 
$\sum a_j=1$ , so for any contraction $T$ on a Banach space we can define the operator
$(I-T)^a =I -\sum_{j=1}^\infty a_jT^j$. The elements of $(I-T)^a X$ were called in \cite{DL} 
{\it fractional coboundaries}. If $T$ is not uniformly ergodic (i.e. $(I-T)X$ not closed), 
then the spaces $(I-T)^aX \subset (I-T)^bX$ for $0<b<a \le 1$ are all different, with closure 
$\overline{(I-T)X}$.

\begin{theo}
Let $T$ and $S$ be commuting mean ergodic contractions on a Banach space $X$ with $F(T)=F(S)$,
and let  $0\le a \le 1$. If $z := (I-T)^a(I-S)^{1-a}x$, then 
$\|\frac1n \sum_{k=0}^{n-1}\sum_{j=0}^{n-1} T^kS^j z \| \to 0$. 
\end{theo}
\begin{proof}
Proposition \ref{degenerate}, which requires $F(T)=F(S)$, yields the extreme cases $a=0$ and $a=1$, 
so we assume $0<a<1$. It was proved in \cite[Corollary 2.15]{DL} That if $T$ is a mean ergodic 
contraction and $y \in (I-T)^aX$, then $\| (1/n^{1-a}) \sum_{k=1}^n T^ky\| \to 0$. It follows that
\begin{equation} \label{uniform-bound}
\sup_{n\ge 1} \Big \| \frac 1{n^{1-a}} \sum_{k=0}^{n-1} T^k(I-T)^a \Big\| = K < \infty.
\end{equation}
If $x \in F(T)$, then $(I-T)^ax=0$, so we may assume $E_Tx=0$, i.e $x \in \overline{(I-T)X}$.
By (\ref{uniform-bound}) and an application of \cite[Corollary 2.15]{DL} to $S$ we obtain
$$
\Big\|\frac1n \sum_{k=0}^{n-1}\sum_{j=0}^{n-1} T^kS^j z \Big\| =
\Big\|\frac1{n^a} \frac1{n^{1-a}}\sum_{k=0}^{n-1}\sum_{j=0}^{n-1} T^kS^j (I-T)^a(I-S)^{1-a}x\Big\| \le
$$
$$
K \Big\| \frac1{n^a} \sum_{j=0}^{n-1}S^j(I-S)^{1-a} x \Big\| \to 0.
$$
Note that for $0<a<1$, the assumption $F(T)=F(S)$ is not used.
\end{proof}

{\bf Remark.} By the theorem, if
$z =\sum_{\ell=1}^L (I-T)^{a_\ell}(I-S)^{1-a_\ell} x_\ell$, with $0\le a_1 <a_2 <\dots a_L \le 1$,
then $\|\frac1n \sum_{k=0}^{n-1}\sum_{j=0}^{n-1} T^kS^j z \| \to 0$.

\begin{cor}
Let $\theta$ and $\tau$ be commuting ergodic measure preserving transformations of the 
probabilty space $(\Omega, \mathcal B,\mathbb P)$, and let $f \in L_2(\mathbb P)$ with 
$\int_\Omega f \ d\mathbb P=0$. Let $T$ and $S$ be the corresponding Koopman operators. If
$$
f =\sum_{\ell=1}^L (I-T)^{a_\ell}(I-S)^{1-a_\ell} g_\ell,
$$
with $0\le a_1 <a_2 <\dots a_L \le 1$  and $g_1,\dots,g_L \in L_2$, then
$\|\frac1n \sum_{k=0}^{n-1}\sum_{j=0}^{n-1} f(\theta^k\tau^j \omega)\|_2 \to 0$. 
\end{cor}

\bigskip

\section{\sc Appendix: The uniform ergodic theorem for commuting contractions}
\medskip

Mbekhta and Vasilescu \cite{MV} extended the uniform ergodic theorem of \cite{L} to $d$ 
commuting operators on a {\it complex} Banach space. A special case of their result is 
the following.
\begin{theo} \label{mv}
Let $T$ and $S$ be commuting power-bounded operators on a {\em complex} Banach space $X$.
Then the following are equivalent:

(i) $(I-T)X+(I-S)X$ and $(I-T^*)X^*+ (I-S^*)X^*$ are closed in $X$ and $X^*$, respectively.

(ii) $\displaystyle{ \frac1{nm}\sum_{k=0}^{n-1}\sum_{j=0}^{m-1} T^kS^j}$ converges in 
operator-norm, as $\min(n,m) \to \infty$.
\end{theo}
The proof in \cite {MV} uses spectral theory, so does not apply directly in real Banach spaces.

\begin{theo} \label{mv-real}
Let $T$ and $S$ be commuting power-bounded operators on a real or complex Banach space $X$.
Then the following are equivalent:

(i) $(I-T)X+(I-S)X$ and $(I-T^*)X^*+ (I-S^*)X^*$ are closed in $X$ and $X^*$, respectively.

(ii) $\displaystyle{ \frac1{nm}\sum_{k=0}^{n-1}\sum_{j=0}^{m-1} T^kS^j}$ converges in 
operator-norm, as $\min(n,m) \to \infty$.

(iii) $\displaystyle{ \frac1{n^2}\sum_{k=0}^{n-1}\sum_{j=0}^{n-1} T^kS^j}$ converges in 
operator-norm, as $n \to \infty$.
\end{theo}
\begin{proof}
Obviously (ii) implies (iii). 
Assume (iii), and put $Y:=\overline{(I-T)X+(I-S)X}$. Then  it 
is easy to compute that ${ \frac1{n^2}\sum_{k=0}^{n-1}\sum_{j=0}^{n-1} T^kS^jz \to 0}$ 
for $z \in Y$ (using power-boundedness), and by (iii), the restrictions to $Y$ satisfy
$\|\frac1{n^2}\sum_{k=0}^{n-1}\sum_{j=0}^{n-1} T^kS^j\|_Y \to 0$.
Hence, for $n$ large enough
$$
A_n:= I_Y - \frac1{n^2} \sum_{k=0}^{n-1}\sum_{j=0}^{n-1} T^kS^j=
 \frac1{n^2} \sum_{k=0}^{n-1}\sum_{j=0}^{n-1}(I- T^kS^j)
$$
is invertible  on $Y$. Since $(I-T^k)(I-S^j) =I-T^k +I-S^j -(I-T^kS^j)$, on $Y$ we have
$$
A_n=\frac1n\sum_{k=0}^{n-1}(I-T^k) +\frac1n\sum_{j=0}^{n-1}(I-S^j) -
\frac1{n^2} \sum_{k=0}^{n-1} \sum_{j=0}^{n-1} (I-T^k)(I-S^j).
$$
Fix $n$ large. Denote by $A$ the restriction to $Y$ of $A_n$ and $B=A^{-1}$ (defined on $Y$).
By the Neumann series expansion, $B$ is in the closed subalgebra  of $B(Y)$ generated by the 
restrictions of $T$ and $S$ to $Y$. These restrictions satisfy
$$
I_Y =BA = B\cdot \frac1n \sum_{k=0}^{n-1}\sum_{\ell=0}^{k-1}T^\ell
\big[I_Y -\frac1n\sum_{j=0}^{n-1}(I_Y-S^j)\big](I_Y-T) +
B\cdot \frac1{n} \sum_{j=0}^{n-1}\sum_{\ell=0}^{j-1} S^\ell(I_Y-S).
$$
Thus we have operators $C$ and $D$ in $B(Y)$, commuting with the restrictions $T_Y$ and $S_Y$, 
such that $I_Y=(I-T_Y)C+(I-S_Y)D$. Hence 
$$
Y =(I-T_Y)Y + (I-S_Y)Y \subset (I-T)X+(I-S)X \subset Y,
$$
so (i) holds. Note that this proof is valid for $X$ real or complex, and  unlike \cite{MV}, 
spectral theory is not used.
\medskip

When $X$ is complex, (i) implies (ii) by \cite{MV}. To prove that (i) implies (ii) when 
$X$ is real we will use the complexification of $X$, described below, and deduce the 
result from the complex case. We define $X_C=X\oplus X$, with the identification 
$(x,y) =x+iy$ which allows the definition of the multiplication by complex scalars. 
On $X_C$ we define the {\it Taylor norm} (see \cite[Proposition 3]{MST})
$$
\|(x,y)\|_T:=\sup_{0\le t \le 2\pi} \|x\cos t -y \sin t\| = 
\sup_{\phi \in X^*, \|\phi\|\le 1} \sqrt{\phi(x)^2 +\phi(y)^2}.
$$
In the sequel we write $\|(x,y)\|$ for $\|(x,y)\|_T$. Note that $\|(x,0)\|=\|x\|$ and
$\|(x,-y)\|=\|(x,y)\|$. Clearly 
$$
\max\{\|x\|,\|y\|\} \le \|(x,y)\| \le \sqrt{\|x\|^2+\|y\|^2} \le \|x\|+\|y\|,
$$
which shows that $\{(x_k,y_k)\}$ converges if and only if both $\{x_k\}$ and $\{y_k\}$ 
converge.  Given an operator $T$ on $X$, we extend it to $X_C$ by $T_C(x,y)=(Tx,Ty)$. By 
\cite[Proposition 4]{MST} $\|(T_C)^n\| = \|(T^n)_C\|=\|T^n\|$, so $T_C$ is 
power-bounded when $T$ is.

Assume now that $T$ and $S$  on $X$ satisfy (i). If $(I-T_C)(x_k,y_k) +(I-S_C)(u_k,v_k)$
converges in $X_C$ to $(z,w)$, computations by the definitions yield that 
$(z,w) \in (I-T_C)X_C+(I-S_C)X_C$, so $(I-T_C)X_C +(I-S_C)X_C$ is closed. 

By \cite[Proposition 7]{MST}, $(X_C)^*$ yields a {\it reasonable} complexification of $X^*$,
which by \cite[Proposition 3]{MST} is equivalent to the Taylor complexification of $X^*$.
It is therefore easy to check that $(T_C)^*=(T^*)_C$, and the condition $(I-T^*)X^*+(I-S^*)X^*$
closed implies that $(I-T_C\,^*)X^*+(I-S_C\,^*)X^*$ is closed. Hence $T_C$ and $S_C$ on $X_C$
satisfy (i), so by Theorem \ref{mv} $\frac1{mn} \sum_{k=0}^{n-1}\sum_{j=0}^{m-1}T_C^kS_C^j$
converges in operator norm on $X_C$, which implies (ii) of our theorem.
\end{proof}

The above proof of (iii) implies (i) yields the following corollary.

\begin{cor} \label{notin-sp}
Let $T$ and $S$ be commuting power-bounded operators on a {\em complex} Banach space $X$.
If $\displaystyle{ \frac1{n^2}\sum_{k=0}^{n-1}\sum_{j=0}^{n-1} T^kS^j}$ converges in 
operator-norm, as $n \to \infty$, then $Y:=(I-T)X+(I-S)X$ is closed, and for any 
subalgebra $\mathcal A \subset B(Y)$ containing the restrictions $T_Y$ and $S_Y$, 
the point $(1,1)$ is not in the Harte joint spectrum $\sigma_{\mathcal A}(T_Y,S_Y)$.
\end{cor}

{\bf Remark.}  The result of Corollary \ref{notin-sp} for the Taylor spectrum was deduced in 
\cite[Lemma 4]{MV} from the spectral mapping theorem \cite{Ta1}; the same proof could apply 
also for the Harte spectrum (in $Y$), by \cite[Theorem 4.3]{Ha}.  Our proof for the Harte 
spectrum is simpler, since it uses only its definition \cite{Ha}.

\bigskip

{\bf Acknowledgements.} The authors are grateful to J.P. Conze, C. Cuny and Y. Derriennic 
for their helpful comments, to I. Assani for discussion of the maximal spectral type,
and to D. Markiewicz for discussion of the different joint spectra.

The authors thank the referee for detecting a gap in the original proof of Theorem \ref{mpt}.


\end{document}